\documentclass[12pt, reqno]{amsart}
\usepackage{amsmath}
\usepackage{amsthm}
\usepackage{amssymb}
\usepackage{amsrefs}
\usepackage{mathrsfs}
\usepackage[usenames]{color}
\usepackage{ifthen}
\usepackage{testhyphens}
\usepackage{graphicx}
\usepackage[all]{xy}
\usepackage{xcolor}

\newtheorem{thm}{}[section]
\newtheorem{theorem}[thm]{Theorem}
\newtheorem{corollary}[thm]{Corollary}
\newtheorem{lemma}[thm]{Lemma}
\newtheorem{proposition}[thm]{Proposition}

\theoremstyle{remark}
\newtheorem{fact}[thm]{\bf Fact}
\newtheorem{remark}[thm]{Remark}

\newtheorem{question}[thm]{Question}
\newtheorem{example}[thm]{Example}

\numberwithin{equation}{section}
\allowdisplaybreaks

\newcommand{\Env}[2][]{%
\ifthenelse{ \equal{#1}{} }
{\ensuremath{#2_{\mathsf{c}}}}
{\ensuremath{#2_{\mathsf{c},#1}}}
}

\newcommand{\II}{\ensuremath{\mathcal{I}}}
\newcommand{\Ind}{\ensuremath{\mathbf{1}}}

\newcommand{\SL}{\ensuremath{\mathcal{L}}}
\newcommand{\ZZ}{\ensuremath{\mathbb{Z}}}
\newcommand{\NN}{\ensuremath{\mathbb{N}}}

\newcommand{\RR}{\ensuremath{\mathbb{R}}}
\newcommand{\FF}{\ensuremath{\mathbb{F}}}
\newcommand{\XX}{\ensuremath{\mathbb{X}}}
\newcommand{\YY}{\ensuremath{\mathbb{Y}}}
\newcommand{\VV}{\ensuremath{\mathbb{V}}}
\newcommand{\WW}{\ensuremath{\mathbb{W}}}
\newcommand{\SSS}{\ensuremath{\mathbb{S}}}
\newcommand{\Leb}{\ensuremath{\mathbf{L}}}

\newcommand{\BB}{\ensuremath{\mathcal{B}}}
\newcommand{\GG}{\ensuremath{\mathcal{G}}}
\newcommand{\XXB}{\ensuremath{\mathcal{X}}}
\newcommand{\LLB}{\ensuremath{\mathcal{Y}}}
\newcommand{\Fou}{\ensuremath{\mathcal{F}}}
\newcommand{\UU}{\ensuremath{\mathcal{U}}}
\newcommand{\Id}{\ensuremath{\mathrm{Id}}}
\newcommand{\kk}{\ensuremath{k}}

\newcommand{\rr}{\ensuremath{\mathbf{r}}}
\newcommand{\xx}{\ensuremath{\mathbf{x}}}
\newcommand{\yy}{\ensuremath{\mathbf{y}}}
\newcommand{\ee}{\ensuremath{\mathbf{e}}}
\newcommand{\zz}{\ensuremath{\mathbf{z}}}

\newcommand{\vv}{\ensuremath{\mathbf{v}}}

\newcommand{\sgn}{\operatorname{sign}}
\newcommand{\supp}{\operatorname{supp}}
\newcommand{\diam}{\operatorname{diam}}
\newcommand{\Ker}{\operatorname{Ker}}

\begin{document}

\title[On certain subspaces of $\ell_p$ for $0<p\le 1$]{On certain subspaces of $\ell_p$ for $0<p\le 1$\\and their applications to  conditional quasi-greedy bases in $p$-Banach spaces}
\author[F. Albiac]{Fernando Albiac}
\address{Mathematics Department-InaMat\\
Universidad P\'ublica de Navarra\\
Campus de Arrosad\'{i}a\\
Pamplona\\
31006 Spain}
\email{fernando.albiac@unavarra.es}

\author[J.~L. Ansorena]{Jos\'e L. Ansorena}
\address{Department of Mathematics and Computer Sciences\\
Universidad de La Rioja\\
Logro\~no\\
26004 Spain}
\email{joseluis.ansorena@unirioja.es}

\author[P. Wojtaszczyk]{Przemys\l{}aw Wojtaszczyk}
\address{Institute of Mathematics of the Polish Academy of Sciences\\
00-656 Warszawa\\
ul. \'Sniadeckich 8\\
Poland}
\email{wojtaszczyk@impan.pl}

\begin{abstract}
We construct for each $0<p\le 1$ an infinite collection of
subspaces of $\ell_p$
that extend the example of Lindenstrauss from \cite{Lin1964} of a subspace of $\ell_{1}$ with no unconditional basis. The structure of this new class of $p$-Banach spaces is analyzed and some applications to the general theory of $\SL_{p}$-spaces for $0<p<1$ are provided. 
The introduction of these spaces serves the purpose to develop the theory of conditional quasi-greedy bases in $p$-Banach spaces for $p<1$. Among the topics we consider are  the existence of infinitely many conditional quasi-greedy bases in the spaces $\ell_{p}$ for $p\le 1$ and
 the careful examination of the conditionality constants of the ``natural basis'' of these spaces.
\end{abstract}

\thanks{F. Albiac acknowledges the support of the Spanish Ministry for Economy and Competitivity under Grant MTM2016-76808-P for \emph{Operators, lattices, and structure of Banach spaces}. F. Albiac and J.~L. Ansorena acknowledge the support of the Spanish Ministry for Science, Innovation, and Universities under Grant PGC2018-095366-B-I00 for \emph{An\'alisis Vectorial, Multilineal y Aproximaci\'on}. P. Wojtaszczyk was supported by National Science Centre, Poland grant UMO-2016/21/B/ST1/00241.}

\maketitle

\section{Introduction}
\noindent  The subject of  finding estimates for the rate of approximation of a function by means of essentially nonlinear algorithms  with respect to biorthogonal systems and, in particular, the greedy approximation algorithm using bases, has attracted much attention for the last twenty years, on the one hand from researchers interested in the applied nature of  non-linear approximation  and, on the other hand from researchers with a more classical Banach space theory background.   Although the basic idea behind the concept of a greedy basis had been around for some time, the formal development of a theory of greedy bases was initiated in 1999 by Konyagin and Temlyakov in the important paper \cite{KoTe1999}. Subsequently, the theory of greedy bases and its derivates  developed very fast as many fundamental results were discovered, and new ramifications branched out. As a result, this is an area with a fruitful interplay between abstract methods from classical Banach space theory and other, more concrete techniques, from approximation theory. 
The reader interested in approximation theory and/or numerical algorithms may consult the paper \cite{Temlyakov2006} or the book \cite{Temlyakov2015}.

 In this article  we will concentrate on the functional analytic aspects of this theory, where, rather unexpectedly, the theory of greedy bases has links to old classical results and also to some open problems. See \cite{AlbiacKalton2016}*{Chapter 10}) for an introductory approach to the subject from this angle.

 To set the mood let us start by recalling the main concepts we will require  from approximation theory in the general setting of quasi-Banach spaces.  Let $\YY$ be a quasi-Banach space and assume that $\BB=(\yy_n)_{n=1}^\infty$ is a semi-normalized fundamental $M$-bounded Markushevich basis in $\YY$, that is, $\BB$ generates the whole space $\YY$ and there is a sequence $(\yy_n^*)_{n=1}^\infty$ in the dual space $\YY^{\ast}$ such that $(\yy_{n}, \yy_{n}^{\ast})_{n=1}^{\infty}$ is a biorthogonal system with $\inf_{n}\Vert \yy_{n}\Vert>0$ and $\sup_{n}\Vert \yy_{n}^{\ast}\Vert<\infty$. From now on we will refer to any such $\BB$ simply as a \emph{basis}. A basic sequence will be a sequence in $\YY$ which is a basis of its closed linear span. Note that semi-normalized Schauder bases are a particular case of bases. Given $A\subseteq\NN$ finite, $S_A=S_{A}[\BB,\YY]\colon\YY\to \YY$ will denote the coordinate projection on $A$ with respect to the basis $\BB$,
\[
S_A(f)=\sum_{n\in A}\yy_{n}^{\ast}(f)\yy_{n}, \quad f\in\YY.
\]
For $f\in \YY$ and $m\in \NN$ we define
\[
\GG_{m}(f)=S_A(f),
\]
where $A\subseteq\NN$ is a set of cardinality $m$ such that $|\yy_{n}^{\ast}(f)|\ge| \yy_{k}^{\ast}(f)|$ whenever $n\in A$ and $k\not\in A$. The set $A$ depends on $f$ and may not be unique; if this happens we take any such set. Thus, the operator $\GG_{m}$ is well-defined, but is not linear nor continuous. The biorthogonal system $(\yy_{n}, \yy_{n}^{\ast})_{n=1}^{\infty}$ (or the basis $\BB$) is said to be \textit{quasi-greedy} provided that there is a constant $C\ge 1$ so that for every $f\in \YY$ and for every $m\in \NN$ we have
\[
\Vert \GG_{m}(f)\Vert \le C\Vert f\Vert.
\]
Equivalently, by \cite{Wo2000}*{Theorem 1}, $(\yy_{n}, \yy_{n}^{\ast})_{n=1}^{\infty}$ is a quasi-greedy system if
\[
\lim_{m\to\infty} \GG_{m}(f)=f\; \text{for each}\; f\in \XX.
\]
Of course, unconditional bases are quasi-greedy, but the converse does not hold in general. Konyagin and Telmyakov provided in \cite{KoTe1999} the first examples of conditional (i.e., not unconditional) quasi-greedy bases. Subsequently, Wojtaszczyk proved in \cite{Wo2000} that, for $1<p<\infty$, the space $\ell_{p}$ has a conditional quasi-greedy basis, and Dilworth and Mitra constructed in \cite{DilworthMitra2001} a conditional quasi-greedy basis of $\ell_{1}$. These papers were the forerunners of an industry devoted to studying the existence of conditional quasi-greedy bases in Banach spaces. The reader will find in the articles \cites{DKK2003, GW2014, AADK, Gogyan2010, Nielsen2007, AAWo, DST2012,DKW2002,BBGHO2018, Wo2000} some of the main achievements in this direction of research. It is important to point out here that not all Banach spaces have a quasi-greedy basis. Indeed, this is the case, for instance, with $\mathcal C([0,1])$ since, by a result of Dilworth et al. \cite{DKK2003} the only $\SL_{\infty}$-space with a quasi-greedy basis is $c_{0}$. (The Banach space old-timers will have made the connection with a classical result of Lindenstrauss and Pe\l czy\'{n}ski \cite{LinPel1968} stating that the canonical $c_{0}$-basis is, up to equivalence, the only unconditional basis of an $\SL_{\infty}$-space.)

From the general point of view of approximation theory, and more specifically the practical implementation of the greedy algorithm for general biorthogonal systems, it is very natural to ask about the existence of conditional quasi-greedy bases in the context of nonlocally convex quasi-Banach spaces. Since $L_{p}([0,1])$ for
$0<p<1$ has trivial dual (making it therefore impossible for $L_{p}$ to have a basis), the first nonlocally convex spaces that come to mind as objects of study for having conditional quasi-greedy bases are the spaces $\ell_{p}$ for $0<p<1$ (see \cite{AABW2019}*{Problem 12.8}). However, the tools that have been developed for building conditional quasi-greedy bases in Banach spaces break down when local convexity is lifted. For instance, the Dilworth-Kalton-Kutzarova method, DKK-method for short, for constructing conditional quasi-greedy bases in a Banach space $\XX$ (\cite{DKK2003}, cf.\ \cite{AADK}) relies on the existence of a complemented subspace $\SSS$ of $\XX$ with a symmetric basis. A careful inspection of the method reveals that the boundedness of the averaging projection with respect to the symmetric basis of $\SSS$ is a key ingredient in the recipe, hence it stops working when $\SSS$ is not locally convex. Since $\ell_p$ for $p<1$ is prime \cite{Stiles1972}, it is hopeless to to try to build a quasi-greedy basis in $\ell_p$ by means of the DKK-method.

These initial drawbacks in making headway create a breeding ground for guesswork. Since quasi-greedy bases in quasi-Banach spaces are not too far from being unconditional (they are unconditional for constant coefficients, see \cite{AABW2019}*{Theorems 3.8 and 3.10}) and the standard unit vector system is the unique, up to equivalence, normalized unconditional basis of $\ell_{p}$ \cite{Kalton1977}, one could be tempted to speculate that it will be the unique quasi-greedy basis in $\ell_{p}$, which would disprove the existence of conditional quasi-greedy bases in the space. In this paper we refute this conjecture and show that indeed such bases of $\ell_{p}$ for $p<1$ exist. In fact, the conditional bases we find belong to the more demanding class of almost greedy bases.

The existence of a conditional quasi-greedy basis of $\ell_p$ shows in particular that $\ell_{p}$ does not have a unique quasi-greedy basis, so we discuss the question of how many mutually non-equivalent quasi-greedy basis there are in $\ell_{p}$, $0<p<1$.

Our construction of conditional almost greedy bases in $\ell_{p}$ for $0<p<1$, is inspired and at the same time extends the example of Dilworth and Mitra from \cite{DilworthMitra2001} of a conditional almost greedy basis in $\ell_{1}$. Their example was derived in turn from the basic sequence constructed by Lindenstrauss \cite{Lin1964} of a monotone, conditional, basic sequence in $\ell_{1}$ whose closed linear span is a $\SL_{1}$-space which is not isomorphic to $\ell_{1}$ and therefore has no unconditional basis. Adapting this script to our context, in Section~\ref{Linpspaces} we manufacture for each $0<p<1$ and each sequence of integers $\delta =(d_{n})_{n=1}^{\infty}$ contained in the interval $[2,\infty)$, a $\SL_{p}$-space denoted $\XX_{p}(\delta)$ which is not isomorphic to $\ell_{p}$; in particular, $\XX_{p}(\delta)$ does not have an unconditional basis. The spaces $\XX_{p}(\delta)$ do have, however, a Schauder basis $\XXB_{p}(\delta)$ whose features are studied in Section~\ref{Linpbases}. We prove that for each $0<p\le 1$ and each sequence $\delta$, the basis $\XXB_{p}(\delta)$ is quasi-greedy and superdemocratic, hence almost greedy. As a by-product of our work we identify the $q$-Banach envelope of the spaces $\XX_{p}$, $0<p< q\le 1$, as being $\ell_{q}$.

To quantify the conditionality of a quasi-greedy basis $\BB$ in a quasi-Banach space $\YY$, in Section~\ref{condisestimates} we study the growth of the constants
\[
k_{m}=\kk_m[\BB,\YY] :=\sup_{|A|\le m} \Vert S_A[\BB,\YY]\Vert, \quad m=1,2.\dots
\]
It follows from a result of Dilworth et al.\ \cite{DKK2003}*{Lemma 8.2} that quasi-greedy bases in Banach spaces cannot be ``too conditional'' in the sense that they satisfy the estimate
\begin{equation}\label{GHO2013ineq}
k_{m}=O(\log m).
\end{equation}
Moreover, there are examples of quasi-greedy bases in certain Banach spaces for which the logarithmic growth is actually attained (see \cite{GHO2013}*{\S 6}). More recently, it was noticed in \cite{AAGHR2015} that there is a entire class of spaces, namely super-reflexive Banach spaces, on which \eqref{GHO2013ineq} can be improved to
\[
k_{m}=O((\log m)^{1-\epsilon})
\]
for some $0<\epsilon <1$. Taking into consideration the role played by the convexity of the space in the proof of inequality \eqref{GHO2013ineq}, it is not surprising that the conditionality constants of quasi-greedy bases in general $p$-Banach spaces satisfy the estimate
\begin{equation}\label{pGHO2013}
k_{m}=O((\log m)^{1/p}).
\end{equation}
We will see this in Corollary~\ref{cor:EstimateCC} after we set in motion the machinery specific to nonlocally convex spaces to prove it. Next, the optimality of \eqref{pGHO2013} within $p$-Banach spaces is established by proving that the reverse inequality, $(\log m)^{1/p}=O(k_{m})$, is attained in $\ell_{p}$ for some conditional quasi-greedy basis. These discussions naturally lead to attempts to construct a conditional quasi-greedy basis $\BB$ with prescribed growth of $k_m$. In Section~\ref{noneqalmostgreedylp} we develop new techniques to produce such bases in $\ell_p$ and $\XX_p$ (see Theorem \ref{thm:twospaces}). Since, roughly speaking, the growth of the sequence $(k_{m})_{m=1}^{\infty}$ is stable under (permutative) equivalence of bases, i.e., $\BB_{1}\sim \BB_{2}$ implies $k_{m}[\BB_{1}]\approx k_{m}[\BB_{2}]$, our result yields the existence of uncountably many (permutatively) non-equivalent quasi-greedy basis in $\ell_{p}$ and $\XX_p$ (Corollary~\ref{dkkplessthan1}). The novelty in our approach has to be seen also in that our techniques are valid for the limit case $p=1$. This nicely complements the main result from \cite{DHK2006}, where Dilworth et al.\ showed that if $1\le p<\infty$ then $\ell_{p}$ has a continuum of permutatively non-equivalent almost greedy bases.

We close in Section~\ref{Lindualbases} with a qualitative and quantitative study of dual Lindenstrauss bases that gives continuity to the results on the subject by Bern\'a et al.\ \cite{BBGHO2018}.

\subsection*{Notation and Terminology.} Throughout this paper we use standard facts and notation from Banach spaces and approximation theory (see e.g.\ \cite{AlbiacKalton2016}). The reader will find the required specialized background and notation on greedy-like bases in quasi-Banach spaces in the recent article \cite{AABW2019}; however a few remarks are in order.

We write $\FF$ for the real or complex scalar field. As is customary, we put $\delta_{k,n}=1$ if $k=n$ and $\delta_{k,n}=0$ otherwise. The unit vector system of $\FF^\NN$ will be denoted by $(\ee_n)_{n=1}^\infty$, i.e., $\ee_n=(\delta_{k,n})_{k=1}^\infty$. For $N\in \NN$, we shall use $S_{m}(x)$ for the projection of $x=(x(n))_{n=1}^{\infty}\in \FF^{\NN}$ onto its first $N$ coordinates, i.e.,
\[
S_{m}(x)= \sum_{n=1}^{m}x(n)\ee_{n}.
\]
and $\supp(x)$ will be the set $\{n\in \NN\colon x(n)\not=0\}$. We will denote by $\langle \cdot,\cdot \rangle$ the natural pairing in $\FF^\NN\times \FF^\NN$, that is, we put
\[
\left\langle (a_n)_{n=1}^\infty, (b_n)_{n=1}^\infty \right\rangle=\sum_{n=1}^\infty a_n b_n
\]
whenever the series $\sum_{n=1}^\infty a_n b_n$ converges.

Given a quasi-Banach space $\YY$ and $\mathcal A\subseteq\YY$, $[\mathcal A]$ denotes the smallest closed subspace of $\YY$ containing $\mathcal A$. Given a Markushevich basis $\BB=(\yy_n)_{n=1}^\infty$ of $\YY$ and $N\in\NN$ we put
\[
\YY^{(N)}[\BB]=[\yy_n \colon 1\le n \le N] \text{ and } \BB^{(N)}=(\yy_n)_{n=1}^N.
\]
If $A\subseteq\NN$ is finite and $\varepsilon=(\varepsilon_n)_{n\in A}$ are signs, we put
\[
\Ind_{\varepsilon,A}=\Ind_{\varepsilon,A}[\BB,\YY]=\sum_{n\in A} \varepsilon_n \yy_n.
\]
If $\varepsilon_n=1$ for all $n\in A$ we denote $\Ind_{\varepsilon,A}$ simply by $\Ind_A$. The same symbol $\Ind_A$ will be used as well to denote the indicator function of a measurable set $A\subseteq [0,1]$.

If $(\BB_n)_{n=1}^\infty$ are Markushevich bases in quasi-Banach spaces $(\YY_n)_{n=1}^\infty$ and $0<q<\infty$,
$\left( \bigoplus_{n=1}^\infty \BB_n\right)_q$ denotes the obvious Markushevich basis in the quasi-Banach space $\left( \bigoplus_{n=1}^\infty \YY_n\right)_p$. If $\BB_n=\BB$ for every $n\in\NN$, we put $\ell_q(\BB)=\left( \bigoplus_{n=1}^\infty \BB_n\right)_q$.

We say that two Markushevich bases $\BB_1=(\yy_j)_{j=1}^\infty$ and $\BB_2=(\zz_j)_{j=1}^\infty$ are equivalent if there is and isomorphism $T$ from $[\BB_1]$ onto $[\BB_2]$ such that $T(\yy_j)=\zz_j$ for all $j\in\NN$. We say that the Markushevich bases $\BB_1$ and $\BB_2$ are permutatively equivalent, and we write $\BB_1\sim \BB_2$, if there is a bijection $\pi\colon\NN\to\NN$ such that $(\yy_{\pi(j)})_{j=1}^\infty$ and $\BB_2$ are equivalent.

The quasi-norm of a linear operator $T$ between two quasi-Banach spaces $\XX$ and $\YY$ will be denoted by $\Vert T\Vert_{\XX\to \YY}$. If $\XX$ and $\YY$ are clear from context, we simply write $\Vert T \Vert=\Vert T\Vert_{\XX\to \YY}$. Note that if $\YY$ is a $p$-Banach space, $0<p\le 1$, then
\begin{equation}\label{eq:lpTopBanach}
\Vert T\Vert_{\ell_p\to \YY}=\sup_{n\in \NN} \Vert T(\ee_n)\Vert_\YY.
\end{equation}

Other more specific notation will be introduced when needed.

\section{Preliminaries on $\SL_p$-spaces for $0<p< 1$}
\noindent $\SL_{p}$-spaces $(1\le p\le \infty)$ were introduced in \cite{LinPel1968} by Lindenstrauss and Pe{\l}czy{\'n}ski as Banach spaces whose local structure resembles that of the spaces $\ell_{p}$. Thus a Banach space $\XX$ is an $\SL_{p}$-space if there is a constant $\lambda$ such that for every finite dimensional subspace $\VV$ of $\XX$ there is a finite dimensional subspace $\WW$ containing $\VV$ and a linear isomorphism $T\colon\WW\to \ell_p^n$ with $\Vert T\Vert \Vert T^{-1}\Vert\le \lambda$.

For $0<p<1$, the theory of $\SL_{p}$-spaces was developed in \cite{Kalton1984} by Kalton, who gave an alternative definition more suitable for $p$-Banach spaces based on the notion of local complementability. Kalton defined a closed subspace $\YY$ of a quasi-Banach space $\XX$ as being \emph{locally complemented} in $\XX$ if there is a constant $\lambda$ such that for every finite-dimensional subspace $\VV$ of $\XX$ and every $\varepsilon>0$ there is a linear operator $T\colon \VV\to \YY$ with $\Vert T \Vert \le \lambda$ and $\Vert T|_{\VV\cap \YY} -\Id_{\VV\cap \YY}\Vert \le\varepsilon$. Then he went on and defined a quasi Banach space to be an \emph{$\SL_p$-space} for $0<p<1$ if it is isomorphic to a locally complemented subspace of $L_p(\mu)$ for some measure $\mu$.
One of the advantages of working with this definition versus adopting the one from the case $p\ge 1$ was that it guarantees that the spaces $L_{p}([0,1])$ for $p<1$ are $\SL_{p}$-spaces.

Let us briefly summarize the relation, depending on the value of $p$, between the classical definition by Lindenstrauss and Pe{\l}czy{\'n}ski of $\SL_{p}$-spaces and that of Kalton's. For $p=1$ both definitions are equivalent. For $1<p<\infty$ the difference is that, with Kalton's definition, Hilbert spaces turn out to be $\SL_{p}$-spaces. For $0<p<1$ its is unknown whether Kalton's definition implies the classical one (see the introductory paragraph of \cite{Kalton1984}*{Section 6}). Nevertheless, the converse in trivially true. Hence, If a $p$-Banach space $\YY$ ($0<p< 1$) possesses an increasing net $(\VV_i)_{i\in I}$ of finite-dimensional subspaces such that $\overline{\cup_{i\in I} \VV_i}=\YY$ and $\VV_i$ is uniformly isomorphic to $\ell_p^{\dim(\VV_i)}$, then $\YY$ is a $\SL_p$-space.

Since Kalton's paper there has been little effort at a systematic treatment of $\SL_{p}$-spaces for $0<p<1$. It is the authors' opinion, however, that these spaces are of interest and therefore deserve such a treatment. The reasons for bringing up $\SL_p$-spaces for $0<p< 1$ here are twofold. Firstly, they provide a very natural general framework for the new $p$-Banach spaces that we will introduce below. Secondly, we will
extend to $p<1$ the classical Lindenstrauss-Pe\l czy\'{n}ski result on unconditional bases in $\SL_1$-spaces \cite{LinPel1968}*{Theorem 6.1} asserting that the only $\SL_{1}$-space with an unconditional basis is $\ell_{1}$.

The proof of Theorem~\ref{thm:UncSLp} relies on the concept of pseudo-dual spaces. Following \cite{Kalton1984}, a quasi-Banach space $\YY$ is said to be a \emph{pseudo-dual} if there is a Hausdorff vector topology on $\YY$ for which the unit ball is relatively compact. By the Banach-Alaoglu theorem, every dual space is a pseudo-dual.

\begin{remark}\label{rmk:lppsd}For every $0<p<\infty$, $\ell_p$ is a separable pseudo-dual space. In fact, every quasi-Banach space $\YY$ with a boundedly complete basis $\BB=(\yy_n)_{n=1}^\infty$ is a pseudo-dual space. To see this, without loss of generality we can assume that $\BB$ is monotone. Then $B_\YY$ is compact with respect to the topology of the coordinate convergence.
\end{remark}

\begin{theorem}\label{thm:UncSLp} Suppose $\YY$ is an $\SL_p$-space ($0<p\le 1$) with an unconditional basis. Then $\YY$ is isomorphic to $\ell_p$.
\end{theorem}

\begin{proof}Assume by contradiction that $\BB$ is an unconditional basis of $\YY$ and that $\YY$ is not isomorphic to $\ell_p$. Then, by \cite{Kalton1984}*{Theorem 6.4}, $\YY$ is isomorphic to a locally complemented subspace of $\ell_p$ which cannot be complemented in $\ell_p$ because $\ell_p$ is prime \cite{Stiles1972}. By \cite{Kalton1984}*{Theorem 4.4}, $\YY$ is not a pseudo-dual space, and so by Remark~\ref{rmk:lppsd}, the basis $\BB$ is not boundedly complete. We deduce the existence of an element $f$ in $\YY$ and of pairwise disjoint sets $(A_n)_{n=1}^\infty$ such that
\[
\inf_n \Vert S_{A_n}[\BB,\YY](f)\Vert>0.
\]
By unconditionality it follows that $(S_{A_n}[\BB,\YY](f))_{n=1}^\infty$ is equivalent to the canonical basis of $c_0$ and so $c_0$ would be isomorphic to a subspace of $\ell_p$, which is an absurdity by Stiles' structural results on $\ell_{p}$ from \cite{Stiles1972}.
\end{proof}

\section{A new family of subspaces of $\ell_p$, $0<p\le 1$.}\label{Linpspaces}
\noindent In 1964, Lindenstrauss \cite{Lin1964} proved the existence of a subspace of $\ell_{1}$, namely $\ker(Q)$, where $Q$ is any bounded linear map from $\ell_{1}$ onto $L_{1}([0,1])$, which is not isomorphic to a dual space.
This space provided the first example of a Banach space with a basis but with no unconditional basis, despite being a subspace of a space having an unconditional basis. This section is devoted to generalizing Lindenstrauss example to the range $0<p<1$. Moreover we will rig our construction in such a way that we produce at once, for every $p<1$, an infinite collection of $p$-spaces fulfilling the desired properties.

The natural way to construct the quotient $Q_p$ from $\ell_p$ onto $L_p[0,1]$ is to start with finite dimensional spaces $(\VV_j)_{j=1}^\infty$ of $L_p[0,1]$ such that $\bigcup_{j=1}^\infty \VV_j$ is dense in $L_p[0,1]$ and each $V_j$ is isometric to a space $\ell_p^{n(j)}$. Then we identify $\ell_p$ with $(\bigoplus_{j=1}^\infty \VV_j)_p$ and for $\vv=(v_j)_{j=1}^\infty$ we define $Q_p(\vv)=\sum_{j=1}^\infty v_j\in L_p[0,1]$.

To control the properties of $Q_p$, in particular to be able to handle $\ker Q_p$ we must be more specific. We start our construction with a sequence of integers $\delta=(d_{n})_{n=1}^{\infty}$ in the interval $[2,\infty)$. Using $\delta$ we will construct a suitable sequence $(h_j)_{j=1}^\infty$ in $B_{L_p[0,1]}$ so that $Q_p(e_j)=h_j$ for $j\in\NN$. The map constructed starting with $\delta$ will be denoted $Q_p(\delta)$.

Before we get started, some terminology is in order. Given an increasing map $\sigma\colon [a,\infty)\cap \ZZ \to \ZZ$ there is a unique non-decreasing map $\rho\colon [\sigma(a),\infty)\to [a,\infty)$ such that $\rho(\sigma(j))=j$ for every $n\in[a,\infty)$ and $\rho(k)\le j$ whenever $k< \sigma(j)$. In fact, $\rho$ can be defined by
\[
\rho(k)=\max\{ n\in[a,\infty) \colon \sigma(n)\le k \}
=\min\{ n\in[a,\infty) \colon k<\sigma(n+1) \}.
\]
We will refer to $\rho$ as to the \emph{left inverse} of $\sigma$. Note also that $\rho(k)=j$ if and only if $k\in[\sigma(j),\sigma(j+1))$. It is straightforward to check that if $\rho_i$ is the left inverse of $\sigma_i$, $i=1$, $2$, then $\rho_1\circ\rho_2$ is the left inverse of $\sigma_2\circ\sigma_1$.
Given $n\in\NN$, and a map $\sigma\colon A\subseteq\NN \to \NN$, we will denote by $\sigma^{(n)}$ the $n$th iteration of $\sigma$, and we will use the convention $\sigma^{(0)}=\Id_\NN.$ The domain of $\sigma^{(n)}$ decreases as $n$ increases. Now we start our construction. For a sequence $\delta=(d_n)_{n=1}^\infty$ with $d_n\geq 2$ for $n=1,2,\dots$ we set
\begin{equation}\label{sigmafunction}
\quad \sigma(k)=2+\sum_{j=1}^{k-1} d_j, \quad k\in\NN,
\end{equation}
and for $A\subseteq\NN$ put
\begin{equation}\label{eq:PhiSets}
\Sigma(A)=\bigcup_{k\in A} [\sigma(k),\sigma(k+1)).
\end{equation}
Since $(\sigma(k))_{k=1}^\infty$ is increasing and $\sigma(1)=2>1$, the sequence
\[
\Lambda(n)=\sigma^{(n)}(1), \quad n\in\NN\cup\{0\}.
\]
is increasing as well. We will also consider the left inverse of $\sigma$,
\[
\rho\colon[2,\infty)\to\NN,\]
and the left inverse of $\Lambda$,
\[
\Gamma\colon \NN\to\NN\cup\{0\}.\]
We have $\Lambda(0)=1$ and $\Lambda(1)=2$. Therefore, $\Gamma(1)=0$ and $\Gamma(2)=1$. Let us define a partition $(J_n)_{n=0}^\infty$ of $\NN$ by
\begin{equation}\label{eq:Jn}
J_n=[\Lambda(n), \Lambda(n+1)), \quad n\ge 0.
\end{equation}

Our construction of the sequence of functions $(h_j)_{j=1}^\infty$ begins here with $
h_1=\Ind_{[0,1)}.$
Note that $\{1\}=J_0$ so the first step produces $\Lambda(1)-1$ functions. Our second step will consist in defining $h_j$ for $j\in J_1=[2,2+d_1)$. Let $(I_j)_{j=2}^{1+d_1}$ be the obvious partition of the unit interval $[0,1)$ into intervals of length $d_1^{-1}$. Set $\lambda_j=d_1^{-1/p}$ for $j\in [2,2+d_1)$ and define
\[
h_{j}=\lambda_j \Ind_{I_j}, \quad j=2,\dots, 1+d_1,
\]
thus in the first two steps we defined $\Lambda(2)-1$ functions.

Let us explain the general recursive process. Assume that $h_j$, $I_j$ and $\lambda_j$ with $\lambda_j^p |I_j|=1$ have been defined for $j<\Lambda(n)$. Given
$
j\in J_n,
$
i.e., $n=\Gamma(j)$, there is unique $k\in J_{n-1}$ such that $j\in [\sigma(k), \sigma(k+1))$, that is $k=\rho(j)$. Let $(I_j)_{j=\sigma(k)}^{\sigma(k)+d_k-1}$ be the obvious partition of $I_k$ into intervals of length $d_k^{-1}|I_k|$.
Define
\[
\lambda_j=d_k^{-1/p}\lambda_k=d_{\rho(j)}^{-1/p}\lambda_{\rho(j)},
\]
and, for $j=\sigma(k),\dots, \sigma(k)+d_k-1$ and $k=\Lambda(n),\dots, \Lambda(n+1)-1$,
\[
h_{j}=\lambda_j \Ind_{I_j}.
\]
Now, for $k\in\NN$ we put
\[
\rr_k=\rr_k(\delta)=-d_k^{-1/p}\sum_{j=\sigma(k)}^{\sigma(k+1)-1} \ee_j,\]
and
\begin{equation}\label{eq:LinBasis}
\xx_k=\xx_k(\delta)=\ee_k-\rr_k.
\end{equation}

Let us summarize some properties of this construction and of $Q_p(\delta)$ that will be used below and that can be directly deduced from the definition.

\begin{fact}\label{fact:a}
For each $n\in\NN\cup\{0\}$, $\II_n=(I_j)_{j\in J_n}$ is a partition on $[0,1)$ into intervals. Moreover $\II_{n+1}$ refines $\II_{n}$ and
\[
\diam(\II_n)=\max\{ |I_j| \colon j\in J_n \} \le 2^{-n}.
\]
Hence, $\lim_n \diam(\II_n)=0$.
\end{fact}

\begin{fact}\label{fact:d} $\Vert Q_p\Vert=1$.
\end{fact}

\begin{fact}\label{fact:b}For each $n\in\NN\cup\{0\}$, $Q_p$ is an isometry from
$
[\ee_j \colon j\in J_n]
$
onto the subspace $L_p^{(n)}$ of $L_p$ consisting of all functions which are constant in each interval of the partition $\II_n$.
\end{fact}

\begin{fact}\label{fact:c} $Q_p(\xx_k)=0$ for every $k\in\NN$.
\end{fact}

\begin{fact}\label{fact:e} For every $k\in\NN$, $\supp(\rr_k)=[\sigma(k),\sigma(k+1))\subseteq[k+1,\infty)$, and $\Vert \rr_k\Vert=1$.
\end{fact}

\begin{fact}\label{fact:q} For every $k\in\NN$, $\supp(\xx_k)=\{k\} \cup[\sigma(k), \sigma(k+1))\subseteq[k,\infty)$, and $\Vert \xx_k\Vert=2^{1/p}$, so from (\ref{eq:PhiSets}) we get $\supp(\sum_{k\in A} a_k\xx_k)\subseteq A\cup \Sigma(A)$ for $A\subset \NN$, finite.
\end{fact}

\begin{fact}\label{fact:r} Since $([\sigma(k), \sigma(k+1)))_{k=1}^\infty$ is a partition of $[2,\infty)$, if $A_1$ and $A_2$ are disjoint subsets of $\NN$, then $\Sigma(A_1)$ and $\Sigma(A_2)$ also are disjoint.
\end{fact}

\begin{fact}\label{fact:w} For each coordinate $j\ge 2$ there are exactly two vectors $\xx_k$ with $\xx_k(j)\not=0$. To be precise, $\xx_j(j)=1$, $\xx_{\rho(j)}(j)=-d_{\rho(j)}^{-1/p}<0$, and $\xx_k(j)=0$ otherwise.
\end{fact}

\begin{fact}\label{fact:y}From Fact~\ref{fact:w} we easily see that for a linear combination $x=\sum_{k=1}^\infty a_k \xx_k$ we have $x(1)=a_1$ and $x(j)=a_j-a_{\rho(j)} d_{\rho(j)}^{-1/p}$ for $j\ge 2$.
\end{fact}

\begin{fact}\label{fact:j}From Fact~\ref{fact:e} we deduce that, if $\supp(x)\subseteq [1,\sigma(k))$, in particular if $x\in [\xx_j \colon 1\le j < k]$, then
\[
\Vert x +\lambda\, \xx_k \Vert^p=\Vert x\Vert^p+|x(k)+\lambda|^p+|\lambda|^p-|x(k)|^p
\]
for any scalar $\lambda$.
\end{fact}

\begin{fact}\label{fact:i}In particular, Fact~\ref{fact:e} yields that, if $\supp(x)\subseteq [1,\sigma(k))$, then $\Vert x\Vert =\Vert x - x(k) \, \xx_k \Vert$.
\end{fact}

Next, we define a sequence of functionals $(\xx_j^*)_{j=1}^\infty$ in $\FF^\NN$. Given $j\in\NN$, let $m\in\NN\cup\{0\}$ be such that $j\in J_m=[\Lambda(m),\Lambda(m+1))$, that is, $m=\Gamma(j)$. Then $\rho^{(n)}(j)$ is defined for $0\le n \le m$, and we have $\rho^{(n)}(j)\in J_{m-n}$. In particular, $\rho^{(m)}(j)=1$. Set

\begin{equation}\label{eq:DualLinBasis}
\xx_j^*=\xx_j^*(\delta)=\ee_j+ \sum_{n=1}^m \left( \prod_{r=1}^n d_{\rho^{(r)}(j)}^{-1/p} \right) \, \ee_{\rho^{(n)}(j)}
=\sum_{n=0}^m \left( \prod_{r=1}^n d_{\rho^{(r)}(j)}^{-1/p} \right) \, \ee_{\rho^{(n)}(j)},
\end{equation}
where, by convention, we put $\sum_{j=1}^0 \lambda_j=0$ and $\prod_{j=1}^0 \lambda_j=1$ for every family $(\lambda_{j})$.
Let us next record some elementary properties of $(\xx_j^*)_{j=1}^\infty$. For $k\in\NN$ and $n\in\NN\cup\{0\}$ put
\begin{equation}\label{eq:Intervals}
J_{k,n}=[\sigma^{(n)}(k) , \sigma^{(n)}(k+1)).
\end{equation}

\begin{fact}\label{fact:h}
Since $\sigma(1)=2$ we have that $J_{1,n}=J_n$ for every $n\in\NN$ and so $(J_{1,n})_{n=0}^\infty$ is a partition of $\NN$. For a general $k\in\NN$, since $k+1\le \sigma(k)$ we have $J_{k,n}<J_{k,n+1}$ for every $n\ge 0$. In particular, $(J_{k,n})_{n=0}^\infty$ are pairwise disjoint integer intervals.
\end{fact}

\begin{fact}\label{fact:s} $\xx_j^*(k)\ge 0$ for every $j$ and $k\in\NN$.
\end{fact}

\begin{fact}\label{fact:k} $\xx_j^*(j)=1$ for every $j\in\NN$.
\end{fact}

\begin{fact}\label{fact:f} $\Vert \xx_j^*\Vert_\infty=1$ for every $j\in\NN$.
\end{fact}

\begin{fact}\label{fact:g} $\supp(\xx_j^*)=\{ \rho^{(n)}(j) \colon 0\le n \le \Gamma(j)\} \subseteq[1,j]$ for every $j\in\NN$.
\end{fact}

\begin{fact}\label{fact:n}$(\xx_j^*)_{j=1}^N$ is a basis of $\{ x \in\FF^\NN \colon \supp(x)\subseteq[1,N]\}$. Hence, the linear span of $(\xx_j^*)_{j=1}^\infty$ is $c_{00}$.
\end{fact}

\begin{fact}\label{fact:l} We infer from Fact~\ref{fact:g} that, for every $j\in\NN$,
\[
\supp(\xx_j^*)\setminus \sigma(\supp(\xx_j^*))=\{ 1 \}.
\]
\end{fact}

\begin{fact}\label{fact:x}If we regard $(\xx_j^*)_{j=1}^\infty$ as an infinite matrix, the ``column'' $(\xx_j^*(k))_{j=1}^\infty$ satisfies
\[
\xx_j^*(k)= \prod_{r=1}^{n} d_{\rho^{(r)}(j)}^{-1/p}
\]
if $j\in J_{k,n}$ for some $n\ge 0$, and $\xx_j^*(k)=0$ if $j\notin \cup_{n=0}^\infty J_{k,n}$.
\end{fact}

\begin{fact}\label{fact:m}We infer from Fact~\ref{fact:x} that $\xx_j^*(\sigma(k))=d_k^{-1/p} \xx_j^*(k)$.
\end{fact}

\begin{fact}\label{fact:p} Also from Fact~\ref{fact:x}, if $j\in J_{k,n}$ then $|\xx_j^*(k)|\le 2^{-n/p}$.
\end{fact}
For $0<p\le 1$ and $\delta$ a sequence of integers in $[2,\infty)$, throughout this paper we will use the following notation
\[\XX_p=\XX_p(\delta) = [\xx_k \colon k\in\NN]\subset \ell_p,\]
\[
\XXB_p= \XXB_p(\delta)=(\xx_k)_{k=1}^\infty,\]
and
\[\XXB_p^*= \XXB_p^*(\delta)=(\xx_k^*)_{k=1}^\infty,
\]
where $\xx_k$ is as in \eqref{eq:LinBasis} and $\xx_k^*$ is as in \eqref{eq:DualLinBasis}.

We will also consider finite-dimensional spaces and finite sequences associated to $\XXB_p(\delta)$. With the convention that $\XX_p^{(0)}(\delta)=\{0\}$, for $n\in\NN$ we define
\[
\XXB_p^{(n)} = \XXB_p^{(n)}(\delta) =(\xx_k)_{k=1}^n.
\]
and
\[
\XX_p^{(n)} =\XX_p^{(n)}(\delta) = \XX_p^{(n)}[ \XXB_p(\delta)]=[\xx_k \colon 1\le k \le n],\]

We will denote by $\XX_p^*=\XX_p^*(\delta)$ the dual space of $\XX_p(\delta)$.

\begin{remark}
The only precedent in the literature for the spaces $\XX_p(\delta)$ and the sequences $\XXB_p(\delta)$ is the case when $p=1$ and $d_n=2$ for $n\in\NN$. The resulting space for those values is indeed the Lindesntrauss space from \cite{Lin1964} to which we referred at the beginning of the section. For other sources alluding to this relevant example, see e.g. \cite{Singer1970}*{Proof of Theorem 15.5}, \cite{LinPel1968}*{Example 8.1}, and \cite{HR1970}.

\end{remark}

\begin{proposition}\label{prop:e} For every $0<p\le 1$ and every sequence $\delta$ of integers in $[2,\infty)$, the pair $(\XXB_p,\XXB_p^*)$ is a biorthogonal system in $\ell_p$, i.e., $\langle \xx_j^*,\xx_k \rangle=\delta_{j,k}$ for $(j,k)\in\NN\times \NN$.
\end{proposition}

\begin{proof}By Facts~\ref{fact:q}, \ref{fact:w}, \ref{fact:g} and \ref{fact:k}, it suffices to consider the case when
$k<j$. Set $n=\Gamma(j)$.

Assume that $k\notin\supp(\xx_j^*)$. Since $\sigma(k)\not=1$, we infer from Fact~\ref{fact:l} that $\sigma(k)\notin \supp(\xx_j^*)$. Then, by Facts~\ref{fact:g} and \ref{fact:r}, $[\sigma(k), \sigma(k+1))$ and $\supp(\xx_j^*)$ are disjoint. Therefore, $\supp(\xx_k)$ and $\supp(\xx_j^*)$ are disjoint so that $\langle \xx_j^*,\xx_k \rangle=0$. Thus we can also assume that $k\in\supp(\xx_j^*)$, i.e., $k=\rho^{(m)}(j)$ for some $1\le m \le n$. This assumption gives $\supp(\xx_j^*)\cap\supp(\xx_k) = \{ \rho^{(m)}(j), \rho^{(m-1)}(j)\}$. Therefore
\[
\langle \xx_j^*,\xx_k \rangle=\prod_{i=1}^m d_{\rho^{(i)}(j)}^{-1/p} - \left( \prod_{i=1}^{m-1} d_{\rho^{(i)}(j)}^{-1/p} \right) d_{\rho^{(m)}(j)}^{-1/p}=0.\qedhere
\]
\end{proof}

\begin{lemma}\label{lem:2}
Let $0<p\le 1$ and $\delta$ be a sequence of integers in $[2,\infty)$. For all positive integers $j$ and $N$ with $j\le N$, there exists $x\in \XX_p^{(N)}$ with $\Vert x \Vert=2^{1/p}$ such that $S_N(x)=\ee_j$.
\end{lemma}

\begin{proof}We proceed by induction on $N$. If $N=j$, put $x=\xx_j$. Assume that the vector $x$ fulfils the desired conditions for some $N\ge j$. Then, by Fact~\ref{fact:i}, $y=x-S_{N+1}(x) \, \xx_{N+1}$ satisfies the condition for $N+1$.
\end{proof}

\begin{lemma}\label{lem:1}Let $0<p\le 1$ and $\delta$ be a sequence of integers in $[2,\infty)$. For each $y\in \ell_p$ and $N\in\NN$ there is $x\in \XX_p^{(N)} $ such that $S_N(y)=S_N(x)$.
\end{lemma}

\begin{proof}It is straightforward from Lemma~\ref{lem:2}.
\end{proof}

\begin{theorem}\label{thm:quotient} Let $0<p\le 1$ and let $\delta$ be a sequence of integers in $[2,\infty)$. Then $Q_p(\delta)$ is a quotient map from $\ell_p$ onto $L_p$ with kernel $\XX_p(\delta)$.
\end{theorem}

\begin{proof}Let us first prove that $Q_p=Q_p(\delta)$ is onto. Let $f\in B_{L_p}$ and $\varepsilon>0$. By Fact~\ref{fact:a}, $\cup_{n=0}^\infty L_p^{(n)}$ is dense in $L_p$ and so there is $n\in\NN$ and $g\in L_p^{(n)}$ with $\Vert f- g\Vert\le \varepsilon$. By Fact~\ref{fact:b} there is $x\in \ell_p$ such that $\Vert x\Vert=\Vert g\Vert$ and $Q_p(x)=g$. Consequently $\Vert f- Q_p(x)\Vert\le \varepsilon$ and $\Vert x \Vert \le (1+\varepsilon^p)^{1/p}$. We infer that
\[
B_{L_p} \subseteq C\overline{Q_p(B_{\ell_p})}
\]
for every $C>1$. By the Open Mapping theorem $Q_p$ is onto, hence a quotient map. In fact, the map from $\ell_p/\Ker(Q_p)$ onto $L_p$ induced by $Q_p$ is an isometry.

By Fact~\ref{fact:c}, $\XX_p=\XX_p(\delta)\subseteq \Ker(Q_p)$. Let us prove the reverse inclusion. Let $x\in\ell_p$ with $Q_p(x)=0$ and fix $\varepsilon>0$. There is $n\in\NN$ such that
\[
\Vert x-S_{\Lambda(n)-1}(x) \Vert \le 2^{-1/p} \varepsilon.
\]
By Lemma~\ref{lem:1} there is $y\in [\xx_k \colon 1\le k <\Lambda(n)]$ such that
\[
z:=S_{\Lambda(n) -1}(y)=S_{\Lambda(n)-1}(x).
\]
Taking into account that
\[
\supp(y-S_{\Lambda(n)-1} (y) )\subseteq [\Lambda(n) , \Lambda(n+1)-1)=J_n
\]
and using Facts~\ref{fact:b}, \ref{fact:c} and \ref{fact:d} we obtain
\[ \|z-y\|=\|Q_p(z-y)\|=\|Q_p(z)\|=\|Q_p(z-x)\|,\]
so that
\[ \|x-y\|^p\leq \|x-z\|^p+\|z-y\|^p\leq2\|x-z\|^p=\varepsilon^p. \qedhere \]
\end{proof}

\begin{proposition}\label{prop:f} For every $0<p\le 1$ and every sequence $\delta$ of integers in $[2,\infty)$ there is a subspace of $\XX_p(\delta)$ isometric to $\ell_p$ and $2^{1/p}$-complemented in $\ell_p$.
\end{proposition}

\begin{proof}
There is a subsequence $(\xx_{k_n})_{n=1}^\infty$ of $\XXB_p(\delta)$ which is a block basic sequence with respect to the unit vector system of $\ell_p$. Indeed, it suffices to choose $k_n=\Lambda(n)$ for every $n\in\NN$.
Let us define maps $J$, $P\colon \ell_p\to \ell_p$ by
\[
J(x)=2^{-1/p} \sum_{k=1}^\infty x(k)\, \xx_{n_k},
\]
and
\[
P(x)=2^{1/p} \sum_{k=1}^\infty \langle \xx_{k_n}^*, S_{A_{k_n}} (x)\rangle \ee_k,\]
where $A_k=\supp(\xx_k)$
By Fact~\ref{fact:e}, $J$ is an isometry. Combining Fact~\ref{fact:f} with the elementary inequality
\begin{equation}\label{eq:lpembedsinl1}
\left|\sum_{i\in I} \alpha_i\right|\le \left(\sum_{i\in I} |\alpha_i|^q\right)^{1/q}, \quad a_i\in\FF, \, q\le 1,
\end{equation}
we have
\[
\Vert P(x)\Vert^p \le 2 \sum_{k=1}^\infty \sum_{j\in A_{n_k}} |\xx_{n_k}^*(j)|^p |x(j)|^p \\
\le 2 \sum_{k=1}^\infty \sum_{j\in A_{n_k}} |x(j)|^p\\
\le 2 \Vert x \Vert^p
\]
for all $x\in\ell_p$. We observe that $J(\ee_j)=\xx_{n_j}$ and $P(\xx_{n_j})=\ee_j$ for all $j\in\NN$. Hence, $P\circ J=\Id_{\ell_p}$.
\end{proof}

The following theorem gathers some structural properties of $\XX_p$-spaces.
\begin{theorem}\label{cor:properties} Let $0<p\le 1$ and $\delta$ be a sequence of integers in $[2,\infty)$. Then:
\begin{enumerate}

\item[(a)] The isomorphic class of $\XX_p(\delta)$ does not depend on the particular choice of $\delta$.
\item[(b)] $\XX_p(\delta)$ is locally complemented in $\ell_p$.
\item[(c)] $\XX_p(\delta)$ is not complemented if $\ell_p$.
\item[(d)] $\XX_p(\delta)$ is an $\SL_p$-space.
\item[(e)] $\XX_p(\delta)$ is not a pseudo-dual space.
\item[(f)] $\XX_p(\delta)$ is not isomorphic to $\ell_p$.
\item[(g)] $\XX_p(\delta)$ does not have an unconditional basis.
\item[(h)] $\XX_p(\delta)$ has a Schauder basis.
\end{enumerate}
\end{theorem}

\begin{proof}
(a) follows from Proposition~\ref{prop:f} and \cite{Kalton1984}*{Theorem 2.2}.
Since $L_p$ is an $\SL_p$-space, (b) follows from \cite{Kalton1984}*{Theorem 6.1.(5)}.
If $\XX_p(\delta)$ were a complemented subspace in $\ell_p$ we would have $\ell_p\simeq \XX_p\oplus \YY$ for some $\YY$ so
\[
\YY\simeq \ell_p/\XX_p \simeq L_p.
\]
Since $L_p$ is not a complemented subspace of $\ell_p$ (in fact, $L_p$ is not a subspace of $\ell_p$) we reach a contradiction, thus (c) holds. Since $\ell_p$ is an $\SL_p$-space as well, (d) holds.
(e) follows from (b) and \cite{Kalton1984}*{Theorem 4.4}.
We deduce (f) from (e) and Remark~\ref{rmk:lppsd}.
(g) is now a consequence of Theorem~\ref{thm:UncSLp}.
Finally, (h) is a consequence of \cite{Kalton1984}*{Theorem 6.4}.
\end{proof}
\begin{remark}It is not too difficult to come up with examples of non-locally convex spaces with a Schauder basis but without an unconditional basis. Indeed, the space $L_1\bigoplus \ell_p$, $0<p<1$, for instance, verifies both conditions. This can be deduced simply by noticing that $L_1$ does not embed in any quasi-Banach space with an unconditional basis. Indeed, the proof for Banach spaces (see, e.g., \cite{AlbiacKalton2016}*{Theorem 6.3.3}) remains valid for quasi-Banach spaces. However, one can argue that adding a locally convex component to a nonlocally convex space is cheating. Thus, the question gains interest if we only accept examples within the class of hereditably non-locally convex spaces. Recall that a quasi-Banach space $\YY$ is \emph{hereditably non-locally convex} if every infinite-dimensional subspace of $\YY$ is non-locally convex, and that a quasi-Banach space $\YY$ is said to be \emph{$\WW$-saturated} if every infinite-dimensional subspace of $\YY$ contains a further subspace isomorphic to $\WW$. Since $\ell_p$ is $\ell_p$-saturated (see \cite{Stiles1972}), every $\SL_p$-space with the bounded approximation property, (BAP) for short, is also $\ell_p$-saturated by \cite{Kalton1984}*{Theorem 6.4}, hence hereditably non-locally convex. Therefore any $\SL_p$-space with (BAP) which is not isomorphic to $\ell_p$ (for instance, the space $\XX_p$ for $p<1$) is an example of a hereditably non-locally convex quasi-Banach space with a Schauder basis but without an unconditional basis. To the best of our knowledge, these are the first-known examples of quasi-Banach spaces with these properties.
\end{remark}

\section{Bases in the spaces $\XX_p$}\label{Linpbases}
\noindent
This section focusses on the sequences $\XXB_p(\delta)$ constructed in the previous section for $p\in(0,1]$ and for any integers $\delta=(d_n)_{n=1}^\infty$ in the interval $[2,\infty)$. Thanks to Theorem~\ref{cor:properties} we know that the space $\XX_p(\delta)$ has a Schauder basis. As a matter of fact, the sequence $\XXB_p(\delta)$ is a Schauder basis of $\XX_p(\delta)$ as we will next prove. Thus, we will rightfully say that $\XXB_p(\delta)$ is the \emph{Lindenstrauss $p$-basis} of $\XX_p(\delta)$.
\begin{proposition}\label{prop:monotonebasis} Given $0<p\le 1$ and a sequence $\delta=(d_n)_{n=1}^\infty$ of integers in $[2,\infty)$, $\XXB_p(\delta)$ is a bi-monotone Schauder basis of $\XX_p(\delta)$.
\end{proposition}

\begin{proof}In order to prove that $\XXB_p(\delta)$ is monotone if suffices to see that $\Vert x\Vert \le \Vert x + a\, \xx_{n}\Vert$ whenever $n\in\NN$ and $x\in\XX_p^{(n-1)}(\delta)$. Let $y= x + a \, \xx_n$. From Fact \ref{fact:e} and Fact \ref{fact:q} we have $x(j)=y(j)$ unless $j=n$, in which case $y(j)=x(j)=a$, or $j\in[\sigma(n),\sigma(n+1))$, in which case $x(j)=0$ and $y(j)=ad_n^{-1/p}$. Therefore
\[
\Vert y\Vert^p - \Vert x\Vert^p = |a+x(j)|^p+|a|^p-|x(j)|^p,
\]
which combined with inequality \eqref{eq:lpembedsinl1} yields $\Vert x \Vert \le \Vert y \Vert$.

Now, to finish the proof we need only show that $\Vert x\Vert \le \Vert x + a\, \xx_{n}\Vert$ whenever $\supp(x)\subseteq[n+1,\infty)$. But if $y$ is as before, a similar argument gives
\begin{align*}
\Vert y \Vert^p-\Vert x\Vert^p&=|a|^p + \sum_{j=\sigma(n)}^{\sigma(n+1)-1} |(x(j) - a d_n^{-1/p}|^p- \sum_{j=\sigma(n)}^{\sigma(n+1)-1}|x(j)|^p\\
&=\sum_{j=\sigma(n)}^{\sigma(n+1)-1} \left|a d_n^{-1/p}\right|^p + |(x(j) - a d_n^{-1/p}|^p-|x(j)|^p\ge 0.\qedhere
\end{align*}
\end{proof}

\begin{remark}
By Proposition~\ref{prop:e}, the sequence $\XXB_p^*=\XXB_p^*(\delta)=(\xx_k^*)_{k=1}^\infty$ regarded inside $\XX_p^*=\XX_p^*(\delta)$ via the natural pairing, is the dual basis of $\XXB_p$. Note that Proposition~\ref{prop:monotonebasis} yields, in particular,
\[\Vert \xx_{k}\Vert \Vert \xx_{k}^{\ast}\Vert \le 1,\quad k\in \NN,\] which combined with Fact~\ref{fact:q}, gives
\[\Vert \xx_k^*\Vert_{\XX_p^*}=2^{-1/p},\quad k\in \NN.\]
\end{remark}

Also thanks to Theorem~\ref{cor:properties} we know that the spaces $\XX_p$ are $\SL_p$-spaces for $0<p\le 1$. In hindsight this can be deduced from Proposition~\ref{Banach-Mazur-distance}.

\begin{proposition}\label{Banach-Mazur-distance}Let $0<p\le 1$ and $\delta$ be a sequence of integers in $[2,\infty)$. For every $n\in\NN$, the Banach-Mazur distance from $\XX_p^{(n)}(\delta)$ to $\ell_p^n$ is not larger than $2^{1/p} $.
\end{proposition}
\begin{proof} For each $n\in\NN$ we will recursively construct vectors $(\xx_{k,n})_{k=1}^n$ in $\XX_p^{(n)}$ such that, if $\rr_{k,n}=\xx_{k,n}-\ee_k$, then $\| \rr_{k,n} \|_p=1$ and $\supp (\rr_{k,n})\subseteq [n+1, \sigma(n+1))$ for $k=1$, \dots, $n$.

For $n=1$ put $\xx_{1,1}=\xx_1$. Let $n\in\NN$ and asume that $(\xx_{k,n})_{k=1}^n$ has been constructed. We set $\xx_{k,n+1}=\xx_{k,n} - \xx_{k,n}(n+1)\xx_{n+1}$ and $\xx_{n+1,n+1}=\xx_{n+1}$. By Fact~\ref{fact:i}, $(\xx_{k,n+1})_{k=1}^{n+1}$ fulfills the desired properties.

We easily infer that the vectors $(\xx_{k,n})_{k=1}^n$ satisfy
\[
\left\|\sum_{k=1}^n a_k\, \xx_{k,n} \right\|^p=\left\|\sum_{j=1}^n a_k\, \ee_k\right\|^p+\left\|\sum_{k=1}^n a_k\, \rr_{k,n} \right\|^p
\]
for every $\alpha=(a_k)_{k=1}^n\in\FF^n$. Hence, \[
\Vert \alpha\Vert \leq \left\|\sum_{k=1}^n a_k\, \xx_{k,n} \right\| \leq 2^{1/p} \Vert \alpha\Vert,\] which proves the proposition.
\end{proof}

Our construction of conditional almost greedy bases in $\ell_p$ will rely on the following isomorphism.
\begin{corollary}\label{cor:isomorphism}Let $0<p\le 1$ and $\delta$ be a sequence of integers in $[2,\infty)$. Then for any sequence of positive integers $(n_k)_{k=1}^\infty$, the space $\left(\bigoplus_{k=1}^\infty \XX_p^{(n_k)}(\delta)\right)_p$ is $2^{1/p}$-isomorphic to $\ell_p$.
\end{corollary}

\begin{proof}By Proposition~\ref{Banach-Mazur-distance}, the infinite direct sum $\left(\bigoplus_{k=1}^\infty \XX_p^{(n_k)}(\delta)\right)_p$ is $2^{1/p}$-isomorphic to $\left(\bigoplus_{k=1}^\infty \ell_p^{n_k}\right)_p$, which, in turn, is isometric to $\ell_p$.
\end{proof}

\subsection{Quasi-greediness of Lindenstrauss $p$-bases, $0<p<1$.}
Our aim in this section is to extend to Lindenstrauss $p$-bases the main result from \cite{DilworthMitra2001}, where it is proved that the Lindenstrauss basis of the space $\XX_{1}(\delta)$, for $\delta$ the constant sequence $d_{n}=2$, in our notation, is conditional and quasi-greedy.

We use $\GG_m=\GG_m[p,\delta](x)$ for the $m$th greedy projection of $x\in\XX_p(\delta)$ with respect to the basis $\XXB_p(\delta)$.

\begin{theorem}\label{thm:QG} For any $0<p\le 1$ and any sequence $\delta=(d_n)_{n=1}^\infty$ of integers in $[2,\infty)$, the Lindenstrauss $p$-basis $\XXB_p(\delta)$ is a quasi-greedy basis of $\XX_p(\delta)$. Quantitatively, for $x\in\XX_p(\delta)$ and $m\in\NN$,
\[
\Vert x-\GG_m[p,\delta](x)\Vert
\le 2^{1/p} \Vert x \Vert,\]
and
\[\Vert \GG_m[p,\delta](x)\Vert \le \min \left\{ 3^{1/p}, \frac{2^{2/p}}{2^{1/p}-1} \right\} \Vert x \Vert.
\]
\end{theorem}

For $p=1$, we get the same estimate as in \cite{DilworthMitra2001}. Note that when $p$ goes to zero the estimate grows as $2^{1/p}$. Before we tackle the proof of this result, we prove a couple of auxiliary lemmas.

\begin{lemma}\label{lem:overlapvectors}Let $0<p\le 1$ and $(\Omega,\Sigma,\mu)$ be a measure space. Suppose $f$, $g\in L_p(\mu)$ are such that $\Vert f \Vert=\Vert g \Vert$. Then,
if $A=\{ \omega\in \Omega \colon f(x) \not = 0\}$,
\[
\Vert f+g \Vert \le 2^{1/p} \Vert (f+g) \Ind_A \Vert.
\]
\end{lemma}
\begin{proof} Since
\begin{align*}
\Vert (f+g) \Ind_{\Omega\setminus A}\Vert^p&=\Vert g \Ind_{\Omega\setminus A}\Vert^p\\
& =\Vert g \Vert^p - \Vert g \Ind_{A}\Vert^p\\
&=\Vert f \Vert^p - \Vert g \Ind_{A}\Vert^p\\
&\leq\Vert (f+g) \Ind_A \Vert^p + \Vert g \Ind_{A}\Vert^p-\Vert g \Ind_{A}\Vert^p,
\end{align*}
applying inequality \eqref{eq:lpembedsinl1} we obtain
\[
\Vert f+g \Vert^p
=\Vert (f+g) \Ind_{A} \Vert^p+\Vert (f+g) \Ind_{\Omega\setminus A}\Vert^p\leq
2\Vert (f+g) \Ind_{A} \Vert^p.\qedhere 
\]
\end{proof}

\begin{lemma}\label{lem:reduction}Let $0<p\le 1$ and $\delta=(d_n)_{n=1}^\infty$ be a sequence of integers in $[2,\infty)$. For every $A\subseteq\NN$ finite and every $x\in[\xx_k \colon k \in A]$ we have $\Vert x\Vert\le 2^{1/p} \Vert S_A(x) \Vert$ and $\Vert x\Vert\le 2^{1/p} \Vert S_{\Sigma(A)}(x) \Vert$.
\end{lemma}
\begin{proof} Set $x=\sum_{k\in A} a_k\, \xx_k$. Let $y=\sum_{k\in A} a_k\, \ee_k$ and $z=\sum_{k\in A} a_k\, \rr_k$, so that $x=y+z$. Since $\Vert y\Vert=\Vert z \Vert<\infty$, $y$ is supported on $A$, and $z$ is supported on $\Sigma(A)$, Lemma~\ref{lem:overlapvectors} yields the desired result.
\end{proof}

\begin{proof}[Proof of Theorem~\ref{thm:QG}] For $k\in\NN$ put $a_k=\langle \xx_k^*, x\rangle$. Let $A$ be the $m$th greedy set of $x$, so that $y:=\GG_m(x)=\sum_{k\in A} a_k\, \xx_k$ and, if $B=\NN\setminus A$,
\[
|a_k|\le |a_j|, \quad k\in B, \, j\in A.
\]
Let $z=x-y=\sum_{k\in B} a_k\, \xx_k$.
By Lemma~\ref{lem:reduction}, it suffices to prove that
\begin{equation}\label{eq:PWqg1}
|y(j)|\le \frac{2^{1/p}}{ 2^{1/p} -1} |x(j)| \ \ \ \mbox{ for every $j\in A$ }
\end{equation}
and that
\begin{equation}\label{eq:PWqg2}
|z(j)|\le |x(j)| \ \ \ \ \ \mbox{ for every $j\in\Sigma(B)$.}
\end{equation}
We split this into three cases:
\begin{itemize}
\item If $j\in A\setminus \Sigma(B)= A\setminus (B\cup \Sigma(B))$ (use Fact \ref{fact:r}), then $z(j)=0$ and, hence, $y(j)=x(j)$.
\item If $j\in \Sigma(B)\setminus A=\Sigma(B)\setminus (A\cup \Sigma(A))$ (use Fact \ref{fact:r}), then $y(j)=0$ and, hence, $z(j)=x(j)$.
\item If $j\in A\cap \Sigma(B) = A \cap ( B\cup \Sigma(B)) = \Sigma(B) \cap (A\cup \Sigma(A))$, (use Fact \ref{fact:r}), then from (\ref{eq:PhiSets}) there is $k\in B$ such that $j\in[\sigma(k),\sigma(k+1))$, that is $k=\rho(j)$.
\end{itemize}
This implies (\ref{eq:PWqg1}) and (\ref{eq:PWqg2}).
\end{proof}
\subsection{Democracy properties of Lindenstrauss $p$-bases.} The concepts of democratic and super-democratic bases are by now fairly standard in greedy approximation theory. They have a verbatim translation into the setting of quasi-Banach spaces (see \cite{AABW2019}*{\S 4})). To quantify the democracy of a basis $\BB$ in a quasi-Banach space $\YY$, we consider the lower and upper democracy functions of $\BB$, given by
\[\Phi_m^l[\BB,\YY]= \inf_{|A|\ge m} \Vert \Ind_A[\BB,\YY]\Vert,\quad m\in\NN,\]
and
\[\Phi_m^u[\BB,\YY]=\sup_{|A|\le m} \Vert \Ind_A[\BB,\YY]\Vert, \quad m\in\NN,\]
respectively. Thus, $\BB$ is democratic if and only if
\[
\Delta[\BB,\YY]=\sup_m \frac{ \Phi_m^u[\BB,\YY]}{ \Phi_m^l[\BB,\YY]}<\infty.
\]
Similarly, the super-democracy of a basis $\BB$ in a quasi-Banach space $\YY$ can be quantified by means of the lower and upper super-democracy functions of $\BB$, which are respectively defined by
\[
\Phi_m^{l,s}[\BB,\YY]= \inf
\{ \Vert \Ind_{\gamma,A}[\BB,\YY]\Vert \colon |A|\ge m,\, \gamma \text{ signs} \}, \quad m\in\NN,\]
and
\[\Phi_m^{u,s}[\BB,\YY]=\sup \{ \Vert \Ind_{\gamma,A}[\BB,\YY]\Vert \colon |A|\le m,\, \gamma \text{ signs}\} ,\quad m\in\NN.
\]
The basis $\BB$ is then super-democratic if and only if
\[
\Delta_s[\BB,\YY]=\sup_m \frac{ \Phi_m^{u,s}[\BB,\YY]}{ \Phi_m^{l,s}[\BB,\YY]}<\infty.
\]
We have $\Phi_m^{u,s}[\BB,\YY]\approx \Phi_m^{u}[\BB,\YY]$ for $m\in\NN$ (see \cite{AABW2019}*{Equation (7.3)}), and it is obvious that if $\YY$ is a $p$-Banach space and $\BB$ is semi-normalized then $ \Phi_m^{u,s}[\BB,\YY]\lesssim m^{1/p}$ for every $m\in\NN$. As we next show, in the case when $\YY=\ell_p$ a reverse inequality holds.

\begin{proposition}Let $0<p<\infty$, let $(\YY_j)_{j=1}^\infty$ be a sequence of finite-dimensional quasi-Banach spaces, and let $\BB=(\yy_n)_{n=1}^\infty$ be a basic sequence in the quasi-Banach space $\YY=(\bigoplus_{j=1}^\infty \YY_j)_p$.
Then
\[ \Phi_m^u[\BB,\YY]\gtrsim m^{1/p}, \quad m\in\NN.
\]
\end{proposition}

\begin{proof}
Given $f=(f_j)_{j=1}^\infty\in\YY$ we will refer to $f_j$ as the $j$th coordinate of $f$. Combining the compactness of each $B_{\YY_j}$ with Cantor's diagonal technique yields $\phi\colon\NN\to\NN$ increasing such that $(\yy_{\phi(k)})_{k=1}^\infty$ converges coordinate-wise. Then $(\yy_{\phi(2k-1)} - \yy_{\phi(2k)})_{k=1}^\infty$ is a null sequence. A gliding-hump argument yields an increasing sequence $\psi\colon\NN\to\NN$ such that $\BB'=(\yy_{\psi(2k-1)} - \yy_{\psi(2k)})_{k=1}^\infty$ is equivalent to a disjointly supported sequence (with respect the coordinates). Therefore, $\BB'$ is equivalent to the unit vector system of $\ell_p$. By \cite{AABW2019}*{Equation~(7.1)}, for $m\in\NN$ we have
\[
\Phi_m^u[\BB,\XX]
\approx \Phi_{2m}^u[\BB,\XX]
\approx \Phi_{2m}^{u,s}[\BB,\XX]
\ge \left\Vert \sum_{k=1}^m \yy_{\psi(2k-1)} - \yy_{\psi(2k)}\right\Vert\\
\approx m^{1/p}.\qedhere
\]
\end{proof}

\begin{proposition}\label{prop:dem}Let $0<p\le 1$ and $\delta=(d_n)_{n=1}^\infty$ be a sequence of integers in $[2,\infty)$. Then the Lindenstrauss $p$-basis $\XXB_p=\XXB_p(\delta)$ is a super-democratic basis of $\XX_{p}=\XX_p(\delta)$.
Quantitatively, for $m\in \NN$:
\begin{enumerate}
\item[(i)] $(1-2^{-1/p}) m ^{1/p}\le \Phi^{l,s}_m[\XXB_p,\XX_p],$ and
\item[(ii)] $\Phi^{u,s}_m[\XXB_p,\XX_p] \le 2 m^{1/p}.$
\end{enumerate}
\end{proposition}

\begin{proof}
Let $A$ be a finite subset of $\NN$ and let $\gamma$ be a sequence of signs. Since $d_k\ge 2$ for every $k\in\NN$, Fact~\ref{fact:y} yields
\[
|\Ind_{\gamma,A}[\XXB_p,\XX_p](j)|\ge 1-2^{-1/p}, \quad j\in A,\]
and so,
\[
\Vert \Ind_{\gamma,A}[\XXB_p,\XX_p]\Vert\ge\left( \sum_{j\in A} \left|\Ind_{\gamma,A}[\XXB_p,\XX_p](j)\right|^p\right)^{1/p} \ge (1-2^{-1/p}) |A|^{1/p}.
\]
This establishes (i).
Inequality (ii) is clear.
\end{proof}

\begin{remark}It is known that all quasi-greedy bases in $\ell_1$ and $\ell_2$ are democratic (see \cite{DST2012}*{Theorem 4.2} and \cite{Wo2000}*{Theorem 3}, respectively) and that if $p\in(1,2)\cup(2,\infty)$ there are quasi-greedy (even, unconditional) bases of $\ell_p$ that are not democratic (see \cite{Pel1960}). We emphasize that the techniques developed to settle the question for $\ell_1$ do not transfer to $\ell_p$ for $0<p<1$, and that all the known examples of quasi-greedy bases in $\ell_p$ are democratic. Thus, the following question seems to be open:
\begin{question}
Is every quasi-greedy basis in $\ell_p$ for $0<p<1$ democratic?
\end{question}
\end{remark}

\begin{corollary}\label{cor:AG+Env} Let $0<p\le 1$ and $\delta=(d_n)_{n=1}^\infty$ be a sequence of integers in $[2,\infty)$.
Then:
\begin{itemize}
\item[(a)] $\XXB_p(\delta)$ is an almost greedy basis of $\XX_p(\delta)$.

\item[(b)] There is a constant $C$ such that for $f\in\XX_p(\delta)$,
\[\Vert (\langle \xx_k^*,f\rangle)_{k=1}^\infty\Vert_{\ell_{p,\infty}} \le C \Vert f \Vert.\]
\end{itemize}
\end{corollary}

\begin{proof} (a) follows by combining \cite{AABW2019}*{Theorem 5.3} (which extends to quasi-Banach spaces the characterization of almost greedy bases in Banach spaces from \cite{DKKT2003}) with Theorem~\ref{thm:QG} and Proposition~\ref{prop:dem}. (b) follows by \cite{AABW2019}*{Theorems 3.13 and 8.12} in combination with Proposition~\ref{prop:dem}.
\end{proof}

\subsection{Banach envelopes of Lindenstrauss $\XX_p$ spaces and bases.} When dealing with a quasi-Banach space $\YY$ it is often convenient to know what the ``smallest'' Banach space containing $\YY$ is (if there is any), or even what the smallest $q$-Banach space containing $\YY$ is. We refer the reader to \cite{AABW2019}*{Section 9} for rigorous definitions and related properties involving these concepts. Our first result in this section exemplifies how tools from (nonlinear) greedy approximation theory can be efficiently used to deduce functional analytic (linear) properties of bases. Indeed, Proposition~\ref{prop:dem} in combination with \cite{AABW2019}*{Proposition 9.12} immediately yield the following result.

\begin{proposition}\label{newpropqenv} Given $0<p< q\le 1$, the $q$-Banach envelope of the basis $\XXB_p(\delta)$ is equivalent to the canonical basis of $\ell_q$. In other words, the $q$-Banach envelope of $\XX_p(\delta)$ is isomorphic to $\ell_q$ under the coefficient transform $\Fou\colon\XX_p(\delta)\to \FF^\NN$ with respect to the basis $\XXB_p(\delta)$.
\end{proposition}
Going further, in this section we will show the following.

\begin{theorem}\label{thm:Env}Let $0<p <q\le 1$ and $\delta=(d_n)_{n=1}^\infty$ be a sequence of integers in $[2,\infty)$. Then the $q$-Banach envelope of $\XX_p(\delta)$ is isomorphic to $\ell_q$ under the inclusion map. In particular, $\XX_p(\delta)$ is dense in $\ell_q$.
\end{theorem}

To tackle the proof of this result, we need a couple of lemmas.
\begin{lemma}\label{lem:blocksum}Let $0<p\le 1$ and $\delta=(d_n)_{n=1}^\infty$ be a sequence of integers in $[2,\infty)$. Then, for every $k\in\NN$, and every $J_{k,n}$ as in (\ref{eq:Intervals}),
\begin{itemize}
\item[(a)]
$\displaystyle \sum_{n=0}^\infty \sup_{j\in J_{k,n}} |x_j^*(k)|^q \le \frac{1}{1-2^{-q/p}}$, for every $0<q<\infty$,
\item[(b)]
$\displaystyle \sum_{j\in J_{k,n}} |\xx_j^*(k)|^p=1$ for every $n\ge 0$.
\end{itemize}
\end{lemma}

\begin{proof}
(a) is a straightforward consequence of Fact~\ref{fact:p}. We will prove (b) by induction on $n$. If $n=0$, $J_{k,n}=\{ k \} $ and, then, the result follows from Fact~\ref{fact:k}. Assume that $n\in\NN$ and that the result holds for $n-1$. Since $\{\sigma(J_{i,n-1}) \colon i\in[\sigma(k), \sigma(k+1))\}$
is a partition of $J_{k,n}$, applying Fact~\ref{fact:m} we obtain
\begin{align*}
\sum_{j\in J_{k,n}} |\xx_j^*(k)|^p&=\sum_{i=\sigma(k)}^{\sigma(k+1)-1} \sum_{k\in J_{i,n-1}} |\xx_j^*(\sigma(k))|^p \\
&=\sum_{i=\sigma(k)}^{\sigma(k+1)-1} d_k^{-1}\sum_{k\in J_{i,n-1}} |\xx_j^*(k)|^p\\
&=\sum_{i=\sigma(k)}^{\sigma(k+1)-1} d_k^{-1}=1.\qedhere
\end{align*}
\end{proof}

\begin{lemma}\label{lem:lqsum}Let $0<p <q \le 1$ and $\delta=(d_n)_{n=1}^\infty$ be a sequence of integers in $[2,\infty)$. Then there is a constant $C$ such that
$\Vert (\xx_j^*(k))_{j=1}^\infty\Vert_q \le C$ for every ``column'' $k\in\NN$.
\end{lemma}

\begin{proof}
Using Lemma~\ref{lem:blocksum} (a) and (b) gives
\begin{align*}
\sum_{j=1}^\infty |\xx_j^*(k)|^q
&= \sum_{n=0}^\infty \sum_{j\in J_{k,n}} |\xx_j^*(k)|^q \\
&\le \sum_{n=0}^\infty \sup_{j\in J_{k,n}} |\xx_j^*(k)|^{q-p} \sum_{j\in J_{k,n}} |\xx_j^*(k)|^p \\
&\le\frac{1}{1-2^{(p-q)/p}}. \qedhere
\end{align*}
\end{proof}

\begin{proof}[Proof of Theorem~\ref{thm:Env}]Define $T\colon \FF^\NN \to \FF^\NN$ by $T(x)=(\langle \xx_j^*, x\rangle)_{j=1}^\infty$. Combining inequality \eqref{eq:lpembedsinl1}, Fubini's theorem, and Lemma~\ref{lem:lqsum} yields
\[
\Vert T(x) \Vert_q^q
\le \sum_{j=1}^\infty \sum_{k=1}^\infty |\xx_j^*(k) x(k)|^q
= \sum_{k=1}^\infty |x(k)|^q \sum_{j=1}^\infty |\xx_j^*(k)|^q \le C^q\Vert x \Vert_q^q,
\]
for all $x\in\FF^\NN$ and some constant $C$.
That is, $T\colon\ell_q \to \ell_q$ is a bounded linear operator that extends the coefficient transform $\Fou\colon\XX_p\to\FF^\NN$ with respect to the basis $\XXB_p(\delta)$ of $\XX_p=\XX_p(\delta)$. Therefore, by Proposition~\ref{newpropqenv}, there is $S\colon\ell_q \to \ell_q$ such that $S\circ\Fou$ is the inclusion map $J$ of $\XX_p$ into $\ell_q$. Pictorially, the diagram
\[
\xymatrix{\ell_q \ar@<+4pt>[r]^S & \ell_q \ar[l]^{T} \\
\XX_p \ar[u]^{\Fou} \ar[ur]_{J} &
}
\]
is commutative. We infer that $T\circ S=\Id_{\ell_q}$. Since $T$ is one-to-one by Fact~\ref{fact:n}, $T$ and $S$ are inverse isomorphisms of one another. Consequently, $\ell_q$ is isomorphic to the Banach envelope of $\XX_p$ under the mapping $S\circ\Fou=J$.
\end{proof}

We close this section with an easy consequence of Theorem~\ref{thm:Env}.
\begin{corollary}\label{cor:Xpnorminginloo}Let $0<p <q\le 1$ and $\delta=(d_n)_{n=1}^\infty$ be a sequence of integers in $[2,\infty)$. Then, for $y\in\FF^\NN$,
\[
\Vert y \Vert_\infty=\sup\{|\langle y, x\rangle| \colon \Vert x\Vert_q \le 1 \}\approx\sup\{|\langle y, x\rangle| \colon \Vert x\Vert_{\XX_p} \le 1\}.
\]
That is, $B_{\XX_p}$ is a norming set for the supremum norm.
\end{corollary}
\begin{proof} Just notice that, by Theorem~\ref{thm:Env}, the dual map $J^*$ of the inclusion map $J$ of $\XX_p(\delta)$ into $\ell_q$ is an isomorphism.
\end{proof}

\begin{question} Suppose $0<p<q\le 1$. Is the $q$-envelope of a $\SL_p$-space a $\SL_q$-space?
\end{question}

\section{Conditionality estimates of quasi-greedy bases in quasi-Banach spaces}\label{condisestimates}

\subsection{Upper bounds for $\kk_m$: the general case.}
This section will be devoted to proving that ``near unconditional'' bases (including quasi-greedy bases) of $p$-Banach spaces satisfy the estimate
\[
k_{m}=O((\log m)^{1/p}).
\]

Borrowing the terminology from \cite{AABW2019} we say that a basis $\BB$ of a quasi-Banach space $\YY$ is \emph{suppression unconditional for constants coefficients} (SUCC for short) if there is a constant $C$ such that
\[
\Vert \Ind_{\gamma,B}[\BB,\YY]\Vert \le C \Vert \Ind_{\gamma,A}[\BB,\YY]\Vert\] whenever $B\subseteq A$ and $\gamma$ is a family of signs.

A crucial ingredient in the study of the vestiges of unconditionality enjoyed by quasi-greedy bases in the setting of nonlocally convex quasi-Banach spaces has been the introduction for $m\in\NN$ of the $m$th \emph{restricted truncation operator} $\UU_m\colon \YY\to \YY$, defined by
\[
\UU_m(f)=\UU_m[\BB,\YY](f)= \min_{j\in A} |\yy_j^*(f)| \sum_{j\in A(m,f)} \sgn(\yy_j^*(f)) \, \yy_j,
\]
where $A(m,f)$ is the $m$th greedy set of $f$ (see \cite{AABW2019}*{\S 3.1}).

\begin{lemma}\label{PWF} If $\sup_n\|\UU_m\|<\infty$ the basis $\BB$ is SUCC.
\end{lemma}
\begin{proof}
Given $A$, $B$ subsets of $\NN$ with $B\subseteq A$, and $\gamma$ a family of signs, for $0<\epsilon <1$ we have
\[
\UU_m(\Ind_{\gamma,A} +\epsilon \Ind_{\gamma,A\setminus B})=\Ind_{\gamma, B},
\]
where $m$ is the cardinality of $B$. This implies that
\[
\|\Ind_{\gamma, B}\|\leq \|\UU_m\| \|\Ind_{\gamma,A}\|.\qedhere
\]
\end{proof}

\begin{lemma}\label{lem:RTO} Let $\BB=(\yy_j)_{j=1}^\infty$ be a basis of a quasi-Banach space $\YY$ for which the restricted truncation operators are uniformly bounded. Then there is a constant $C$ such that
\begin{equation*}
\left\Vert \sum_{j\in B} b_j \, \yy_j\right\Vert \le C \Vert f \Vert
\end{equation*}
whenever $B\subseteq\NN$ is finite and $ |b_j| \le |\yy_k^*(f)|$ for every $j$, $k\in B$.
\end{lemma}
\begin{proof}Assume without loss of generality that $y= \sum_{j\in B} b_j \, \yy_j\not=0$, i.e., $\lambda=\max_{j\in B} |b_j|>0$.
Let
\[
A=\{k\in\NN \colon |\yy_k^*(f)| \ge \lambda\}.
\]
Since $\BB$ is SUCC by Lemma~\ref{PWF}, and $B\subseteq A$, \cite{AABW2019}*{Lemma 2.2} yields
\[
\Vert y \Vert \le C \lambda \left\Vert \sum_{j\in A} \sgn(\yy_j^*(f)) \, \yy_j\right\Vert
\]
for some constant $C$ that only depends on $\BB$. If $|B|=m$, $B$ is the $m$th greedy set of $f$. Since $\lambda\le \min_{k\in B} |\yy_k^*(f)| $,
$\Vert y \Vert\le C \Vert \UU_m(f) \Vert$.
\end{proof}

\begin{theorem}\label{thm:EstimateCC}Let $0<p\le 1$ and $\BB$ be a basis of a $p$-Banach space $\XX$ for which the restricted truncation operators are uniformly bounded. Then
\begin{equation}\label{AAWDKKInequality}
\kk_m[\BB,\XX]\lesssim \log^{1/p}(m) , \quad m\ge 2.
\end{equation}
\end{theorem}
\begin{proof} Let $c=\sup_j \Vert \yy_j^*\Vert<\infty$ and $d=\sup_j \Vert \yy_j\Vert<\infty$. For $f\in\YY$ put
\[
B_n=\{ j\in \NN \colon |\yy_j^*(f)|\le c \Vert f \Vert 2^{-n/p} \}, \quad n\in\{0\}\cup\NN.
\]
Of course, $B_0=\NN$. Let $m\in \NN$ and pick $N\in\NN\cup\{0\}$ such that $2^N\le m <2^{N+1}$. For
$A\subseteq\NN$ with $|A|\le m$ we consider the partition $(A_n)_{n=0}^{N}$ of $A$ given by
\[
A_n=A\cap (B_n\setminus B_{n+1})\;\text{if}\; n=0,\dots, N-1, \;\text{and}\; A_{N}=A\cap B_{N}.
\]
Note that if $j$, $k\in A_n$ for some $n\le N-1$, then
\[
|\yy_j^*(f)| \le c \Vert f \Vert 2^{-n/p} \le 2^{1/p} |\yy_k^*(f)|= |\yy_k^*( 2^{1/p} f)|.
\]
By Lemma~\ref{lem:RTO},
\[
\Vert S_{A_n}[\BB,\YY](f)\Vert \le 2^{1/p}\ C \Vert f \Vert, \quad n=0, \dots, N-1.
\]
As for the set $A_N$, we have $|A_N|\le |A|\le m<2^{N+1}$ and
\[
\Vert \yy_j^*(f) \, \yy_j\Vert \le cd 2^{-N/p}, \quad j\in A_N.
\]
Using $p$-convexity we obtain
\[
\Vert S_A[\BB,\YY](f)\Vert^p \le ( c^p d^p |A_{N}| 2^{-N} + N C^p)\Vert f \Vert^p < (2c^p d^p+2N C^p) \Vert f \Vert^p.
\]
Hence,
\[\kk_m[\BB,\YY]\le 2^{1/p}(c^p d^p+ C^p \log_2(m))^{1/p}.\qedhere\]
\end{proof}

\begin{corollary}\label{cor:EstimateCC} If $\BB$ is a quasi-greedy basis of a $p$-Banach space $\XX$, then $\kk_m[\BB,\XX]\lesssim \log^{1/p}(m)$ for $m\ge 2$.
\end{corollary}

\begin{proof}It is straightforward from \cite{AABW2019}*{Theorem 3.13} and Theorem~\ref{thm:EstimateCC}.
\end{proof}

\subsection{An upper bound for $\kk_m[\XXB_{p},\XX_p], 0<p<1$.}
Now we concentrate in the case when the sequence $\delta=(d_{n})_{n=1}^{\infty}$ involved in the construction of $\XXB_p(\delta)$ and $\XXB_p(\delta)^*$ is non-decreasing, so that the function $\sigma$ defined in \eqref{sigmafunction} is \emph{convex}, i.e., $(\sigma(n+1)-\sigma(n))_{n=1}^\infty$ is non-decreasing. Notice that any convex function $\sigma\colon \NN\to \NN$ satisfies the inequality
\begin{equation}\label{eq:convex}
\frac{\sigma(k)-\sigma(j)}{k-j} \le \frac{\sigma(k')-\sigma(j')}{k'-j'},
\quad j<k,\,j'<k',\, k\le k', \, j \le j'.
\end{equation}
Hence, $\sigma$ is convex and non-decreasing if and only if
\[
0\le \sigma(k)-\sigma(k') \le \sigma(k'')-\sigma(k) , \quad 0\le k - k' \le k''-k.
\]
Iterating this formula yields that if $\sigma_1$ and $\sigma_2$ are convex and non-decreasing so is $\sigma_1\circ\sigma_2$.

We start by proving some additional properties of the integer intervals $J_{k,n}=[\sigma^{(n)}(k) , \sigma^{(n)}(k+1))$ defined in \eqref{eq:Intervals}.

\begin{lemma}\label{Iproperties}
Let $0<p\le 1$ and $\delta=(d_n)_{n=1}^\infty$ be a non-decreasing sequence of integers in $[2,\infty)$. For every $k\ge 1$ and every $n\in\NN\setminus\{0\}$ we have:
\begin{itemize}
\item[(a)] $(x_j^*(k))_{j\in J_{k,n}}$ is non-increasing;
\item[(b)] $\min \{ |x_j^*(k)| \colon j \in J_{k,n} \} \geq 2^{1/p} \max \{ |x_j^*(k)| \colon j \in J_{k,n+1}\}$;
\item[(c)] $ |J_{k,n}| \le |J_{k+1,n}|$ for every $n\ge 0$.
\end{itemize}
\end{lemma}
\begin{proof}
(a) follows from Fact~\ref{fact:x} taking into account that $\rho$ is non-decreasing. In order prove (b) we pick $j\in I_{k,n}$ and $i\in I_{k+1,n}$. Since $j<i$ by Fact~\ref{fact:h}, the same argument yields
\[
\xx_i^*(k) = \prod_{r=1}^{n+1} d_{\rho^{(r)}(i)}^{-1/p} \le 2^{-1/p} \prod_{r=1}^{n} d_{\rho^{(r)}(i)}^{-1/p} \le 2^{-1/p} \prod_{r=1}^{n} d_{\rho^{(r)}(j)}^{-1/p} = 2^{-1/p} \xx_j^*(k).
\]
(c) follows from the fact that $\sigma^{(n)}$ is convex.
\end{proof}

Given $A\subseteq\NN$ finite, let us consider the linear map $P_A\colon \FF^\NN \to \FF^\NN$ defined by
\[
P_A(f)=\sum_{k\in A} \langle \xx_k^*, f \rangle \, \xx_k.
\]
The restriction of $P_A$ to $\XX_p$ is the coordinate projection $S_A[\XXB_p,\XX_p]$ on the set $A$ with respect to the basis $\XXB_p$. Let us also consider the auxiliary linear operator $T_A\colon \FF^\NN \to \FF^\NN$ given by
\[
T_A(f)=\sum_{k\in A} \langle \xx_k^*, f \rangle \, \ee_k.
\]
\begin{lemma} \label{LemmaOperators} Let $0<p\le 1$ and $\delta=(d_n)_{n=1}^\infty$ be a sequence of integers in $[2,\infty)$. For any finite set $A\subseteq \NN$ we have
\[
\Vert S_A[\XXB_p,\XX_p]\Vert
\leq \|P_A\|_{\ell_p\to \ell_p} \leq 2^{1/p}\|T_A\|_{\ell_p\to\ell_p}.
\]
\end{lemma}
\begin{proof}The left hand-side inequality is obvious. As for the inequality on the right, it suffices to note that $P_A=R\circ T_A$ where $R:\ell_p\rightarrow \ell_p$ is a linear operator such that $R(\ee_k)=\xx_k$ for every $k\in \NN$. Fact \ref{fact:q} implies that $\|R\|=2^{1/p}$.
\end{proof}

\begin{lemma}\label{lem:lpnormcolumn}Let $0<p\le 1$ and $\delta=(d_n)_{n=1}^\infty$ be a non-decreasing sequence of integers in $[2,\infty)$. Then,
\begin{enumerate}
\item[(i)] $
\sum_{j\in A} |\xx_j^*(k)|^p\le 1+\Gamma(m)$ for every $A\subseteq\NN$ with $|A|\le m $.
\item[(ii)] $\sum_{j=1}^m |\xx_j^*(1)|^p>\Gamma(m)$ for every $m\in\NN$.
\end{enumerate}
\end{lemma}

\begin{proof}Let $N=N(k,m)$ be the smallest positive integer such that
$m\le\sum_{n=0}^N |J_{k,n}|.$
Using Fact~\ref{fact:h}, Lemma~\ref{Iproperties}, and Lemma~\ref{lem:blocksum}(a), if $A'=\cup_{n=0}^N J_{k,n}$,
\[
\sum_{j\in A} |\xx_j^*(k)|^p\le \sum_{j\in A'} |\xx_j^*(k)|^p
=\sum_{n=0}^N \sum_{j \in J_{k,n}} |\xx_j^*(k)|^p=1+N.
\]
By Lemma~\ref{Iproperties}(c), $N\le N(1,m)$. Moreover, by Fact~\ref{fact:h},
\[
N(1,m)=\min\{ n \in\NN \colon m<\Lambda(n+1) \}=\Gamma(m).
\]
Finally, we note that the arguments we have used also give
\[
\sum_{j=1}^m |\xx_j^*(1)|^p>\sum_{n=0}^{N(1,m)-1} \sum_{j \in J_{1,n}} |\xx_j^*(k)|^p=N(1,m).\qedhere
\]
\end{proof}

\begin{proposition}\label{prop:upper}Let $0<p\le 1$ and $\delta=(d_n)_{n=1}^\infty$ be a non-decreasing sequence of integers in $[2,\infty)$. Then,
\[
\kk_m[\XXB_p,\XX_p] \le 2^{1/p} (1+\Gamma(m))^{1/p}, \quad m\in\NN.
\]
\end{proposition}
\begin{proof}In light of Lemma~\ref{LemmaOperators} and equation~\eqref{eq:lpTopBanach}, it suffices to prove that $\Vert T_A(\ee_k)\Vert \le (1+\Gamma(m))^{1/p}$ whenever $k\in\NN$ and $|A|\le m$. Since
\[
\Vert T_A(\ee_k)\Vert^p=\sum_{j\in A} |\xx_j^*(k)|^p,
\]
Lemma~\ref{lem:lpnormcolumn} provides the desired estimate.
\end{proof}

\subsection{Optimality of the upper bound for $\kk_m[\XXB_{p},\XX_p], 0<p<1$.}
Our next task is to show that the estimate provided by Proposition~\ref{prop:upper} is optimal. Like in \cite{AAWo} we will consider the alternative conditionality constants of a basis $\BB=(\yy_n)_{n=1}^\infty$ of a quasi-Banach space $\YY$, defined for $m\in\NN$ by
\[
\widetilde{\kk}_m[\BB,\YY]=\sup \{ \Vert S_A[\BB,\YY](f)\Vert \colon A\subseteq\NN,\, f\in\YY^{(m)}[\BB], \, \Vert f \Vert\le 1\}.
\]
These constants depend on the particular ordering we choose for the basis, while $ \kk_m[\BB,\YY]$ do not.
It is obvious that $\widetilde{\kk}_m[\BB,\YY]\le \kk_m[\BB,\YY]$ for every $m\in\NN$.

\begin{proposition}\label{prop:lower}Let $0<p\le 1$ and $\delta=(d_n)_{n=1}^\infty$ be a sequence of integers in $[2,\infty)$. Then,
\[
2^{-2/p} (1+\Gamma(m))^{1/p} \le \widetilde{\kk}_m[\XXB_p,\XX_p] , \quad m\in\NN.
\]
\end{proposition}

\begin{proof}
We recursively define $(u_k)_{k=1}^\infty$ and $(v_k)_{k=1}^\infty$ in $\XX_p$. We start with $u_1=v_1=\xx_1$. Assuming that  $u_k$ and $v_k$ have been  constructed for $k\in \mathbb N$, we define
\begin{align*}
u_{k+1}&=u_{k} - u_{k}(k+1) \, \xx_{k+1},\\
v_{k+1}&=v_{k} - \sgn(v_k(k+1)) \, u_{k}(k+1) \, \xx_{k+1}.
\end{align*}
Using Fact~\ref{fact:e}, by induction on $k\in\NN$ we obtain that
\begin{itemize}
\item[(i)] $u_k$ and $v_k$ are linear combinations of $\XXB_p^{(k)}$ with real scalars,
\item[(ii)] $\supp(u_k) \subseteq \{1\}\cup[k+1,\sigma(k+1)) $, and
\item[(iii)] $u_{k}(1)=1$
\end{itemize}
Whence, by Proposition~\ref{prop:e},
\begin{itemize}
\item[(iv)] $\langle \xx_{k+1}^*,u_{k+1}\rangle=-u_k(k+1)$ for every $k\in\NN$, and
\item[(v)] $|\langle\xx^*_s,u_k\rangle|=|\langle\xx^*_s,v_k\rangle|$ for every $k$, $s\in \NN$.
\end{itemize}
From Fact~\ref{fact:i} we get $\Vert u_k\Vert=2^{1/p}$ for every $k\in\NN$. Now we aim at estimating $\|v_k\|$ from below.
By Fact~\ref{fact:g} and (ii), (iii) and (iv), for every $k\in\NN$ we have
\[
u_k(k+1)=-\langle\xx_{k+1}^*,u_{k+1}\rangle=-\xx_{k+1}^*(1) \, u_{k+1}(1)=-\xx_{k+1}^*(1).
\]
Therefore, by Fact~\ref{fact:s}, $u_k(k+1)\le 0$. Then, applying Fact~\ref{fact:j}, for  $k\in\NN$ we get
\begin{align*}
\Vert v_{k+1}\Vert^p-\Vert v_{k}\Vert^p
&=(|v_{k}(k+1)|+|u_{k}(k+1)|)^p+|u_{k}(k+1)|^p-|v_{k}(k+1)|^p\\
&\ge |u_{k}(k+1)|^p
=|\xx_{k+1}^*(1)|^p.
\end{align*}
Combining this inequality with Lemma~\ref{lem:lpnormcolumn} yields
\begin{align*}
\Vert v_{m}\Vert^p&= \Vert v_1\Vert^p+\sum_{k=2}^{m} \Vert v_{k}\Vert^p-\Vert v_{k-1}\Vert^p\\
&\ge 2+\sum_{k=2}^m |\xx_{k}^*(1)|^p\\
&=1+\sum_{k=1}^m |\xx_{k}^*(1)|^p\\
&>1+\Gamma(m),
\end{align*}
for every $m\in\NN$. Finally, by (i) and (v),
\[
\widetilde{\kk}_m[\XXB_p,\XX_p]\ge 2^{-1/p} \frac{ \Vert v_k\Vert}{\Vert u_k\Vert} \ge 2^{-2/p} (1+\Gamma(m))^{1/p}. \qedhere
\]
\end{proof}
Putting together Propositions~\ref{prop:upper} and \ref{prop:lower} we can state the following theorem.

\begin{theorem}\label{thm:main}Let $0<p\le 1$ and $\delta=(d_n)_{n=1}^\infty$ be a non-decreasing sequence of integers in $[2,\infty)$. Then,
\[
\kk_m[\XXB_p,\XX_p] \approx \widetilde{\kk}_m[\XXB_p,\XX_p]\approx \Gamma^{1/p}(m), \quad m\ge 2.
\]
\end{theorem}

\subsection{Lebesgue constant estimates for Lindenstrauss $p$-bases}
We close this section with an estimate for the performance of the greedy algorithm implemented in the space $\XX_p(\delta)$ with respect to the Lindenstrauss $p$-basis $\XXB_p(\delta)$. To put this small addition in context, we recall that for $m\in \NN$, the \emph{best $m$-term approximation error} of $f\in \YY$ with respect to $\BB$ is given by
\[
\sigma_{m}(f)=\sigma_{m}(f;\BB, \YY):=\inf\left\{\left\Vert f- z\right\Vert \right\},
\]
the infimum being taken over all $m$-term linear combinations $z$ of vectors from $\BB$. A relevant question in the literature raised by Temlyakov at the turn of the century, is to compare the error in the approximation of $f$ by $\GG_{m}(f)$, measured by $\Vert f- \GG_m(f) \Vert$, with $\sigma_{m}(f)$. For a fixed basis $\BB$ in $\YY$ and $m\in \NN$, the $m$th \emph{Lebesgue constant}, $\Leb_m[\BB,\YY]$, is the smallest constant $C$ such that
\[
\Vert f- \GG_m(f) \Vert \le \Leb_m[\BB,\YY] \sigma_{m}(f),
\]
for all $f\in \YY$. This is sometimes referred to as a \emph{Lebesgue-type inequality for the greedy algorithm}.

The growth of the Lebesgue constants as $m$ increases has been studied in \cites{GHO2013,BBGHO2018} in the framework of Banach spaces. As for non-locally convex quasi-Banach spaces, let us point out that, if we put
\begin{equation*}
\kk_m^c[\BB,\YY] =\sup_{|A|\le m} \Vert\Id_\YY- S_A[\BB,\YY]\Vert, \quad m\in\NN,
\end{equation*}
the proof of \cite{AABW2019}*{Theorem 6.2} yields that for a super-democratic basis $\BB$ of a $p$-Banach space $\YY$ we have
\begin{equation}\label{eq:LebEstimates}
\kk_m^c[\BB,\YY] \le \Leb_m[\BB,\YY] \le \left(1 + \frac{ \Delta_s^p[\BB,\YY]}{(2^p-1)^{2}}\right)^{1/p} \kk_m^c[\BB,\YY]
\end{equation}
for every $m\in\NN$. Note also that $\kk_m^c[\BB,\YY]\approx \kk_m[\BB,\YY]$ for $m\in\NN$.
\begin{corollary}Let $0<p\le 1$ and $\delta$ be a non-decreasing sequence of integers in $[2,\infty)$. Then
\[
\Leb_m[\XXB_p(\delta),\XX_{p}(\delta) ] \approx \Gamma^{1/p}(m), \quad m\ge 2.
\]
\end{corollary}

\begin{proof}Just combine \eqref{eq:LebEstimates} with Theorem~\ref{thm:main} and Proposition~\ref{prop:dem}.
\end{proof}

\section{Non-equivalent almost-greedy bases in $\ell_p$ and $\XX_p$, $0<p\le 1$.}\label{noneqalmostgreedylp}

\noindent In this section we give a neat application to the structure of the spaces $\ell_{p}$, $0<p\le 1$, in that they contain an uncountable set of mutually non-equivalent (conditional) almost greedy bases. In fact, these bases are not even permutatively equivalent. We will construct each of these bases in a space isomorphic to $\ell_{p}$ instead of in $\ell_{p}$ itself.

Let $0<p\le 1$. Given a sequence of integers $\delta$ in $[2,\infty)$ and an unbounded sequence of positive integers $\eta=(N_k)_{k=1}^\infty$ we consider the following direct sum of finite-dimensional Lindenstrauss $p$-bases:
\[
\LLB_p[\delta,\eta]=\left(\bigoplus_{k=1}^\infty \XXB_p^{(N_k)}(\delta)\right)_p.
\]

\begin{lemma}\label{lem:doubling}Let $\delta$ be a sequence of integers in $[2,\infty)$. Then the sequence $\Gamma$ is doubling. To be precise, it satisfies
\[
\Gamma(2m)\le \Gamma(m)+1 \le 2\Gamma(m), \quad m\ge 2.
\]
\end{lemma}
\begin{proof}It is clear by definition that $2k\le \sigma(k)$ for every $k\in\NN$. Then, if $n=\Gamma(m)$ we have $2m<2\Lambda(n+1)\le \Lambda(n+2)$ so that $\Gamma(2m)\le n+1$.
\end{proof}
\begin{theorem}\label{thm:Baseslp}
Let $0<p\le 1$, $\delta$ be a non-decreasing sequence of integers in $[2,\infty)$ and $\eta$ an unbounded sequence of positive integers. Then $\LLB_p[\delta,\eta]$ is an almost greedy Schauder basis of a space isomorphic to $\ell_p$. Moreover
\[
\Leb_m[\LLB_p[\delta,\eta],\ell_p] \approx \kk_m[\LLB_p[\delta,\eta],\ell_p] \approx \Gamma^{1/p}(m), \quad m\ge 2.
\]
In the case when $\sum_{j=1}^k N_k \lesssim N_{k+1}$ for $k\in\NN$, we also have
\[
\widetilde{\kk}_m[\LLB_p[\delta,\eta],\ell_p] \approx \Gamma^{1/p}(m), \quad m\ge 2.
\]
\end{theorem}
\begin{proof}
Combine Lemma~\ref{lem:doubling}, \cite{AADK}*{Lemma 2.3}, Proposition~\ref{prop:monotonebasis}, Corollary~\ref{cor:isomorphism}, Theorem~\ref{thm:QG}, Proposition~\ref{prop:dem} and Theorem~\ref{thm:main} (see also \cite{AABW2019}*{$\S$ 10.2}).
\end{proof}
We emphasize that, as we next show, the sequence $\eta$ chosen for building the direct sum plays no significant role.
We say that a subbasis of a Markushevich basis is complemented if the coordinate projection onto the subspace generated by the subbasis is bounded. Of course, in the lack of unconditionality, there are subbases that are not complemented so it seems hopeless trying to extend to Markushevich bases the Schr\"oder-Bernstein theorem for unconditional bases (see \cite{Wo97}*{Proposition 2.11}). In this situation the decomposition method comes to our aid.

The proof of the decomposition method for Markushevich bases in quasi-Banach spaces stated in Lemma~\ref{lem:PDC} is similar to the proof of Pe\l czy\'{n}ski's decomposition method for Banach spaces (see \cite{Pel1969} or \cite{AlbiacKalton2016}*{Theorem 2.2.3}) so we omit it. If $\BB_1$ is permutatively equivalent to a complemented subbasis of $\BB_2$, we write $\BB_1\lesssim_c \BB_2$.

\begin{lemma}\label{lem:PDC}Let $\BB_1$ and $\BB_2$ be Markushevich bases of quasi-Banach spaces. Assume that $\BB_1\lesssim_c \BB_2$, that $\BB_2\lesssim_c \BB_1$ and that, for some $0<q\le \infty$, $\ell_q(\BB_1)\sim\BB_1$ (we replace $\ell_q$ with $c_0$ if $q=\infty$).Then $\BB_1\sim \BB_2$.
\end{lemma}
\begin{proposition}\label{prop:equivalence}Let $0<p\le 1$, $\delta$ be a non-decreasing sequence of integers in $[2,\infty)$, and $\eta$ and $\eta'$ be unbounded sequences of positive integers. Then
$
\LLB_p[\delta,\eta]\sim \LLB_p[\delta,\eta'].
$
\end{proposition}

\begin{proof}It suffices to prove the statement in the case when $\eta=(n_k)_{k=1}^\infty$ is ``universal'', i.e.,
\[
|\{ k\in\NN \colon n_k=m\}|=\infty
\]
for every $m\in\NN$. Then we have $\LLB_p[\delta,\eta']\lesssim_c \LLB_p[\delta,\eta]$ and $\ell_p(\LLB_p[\delta,\eta])\sim\LLB_p[\delta,\eta]$. Since $\eta'$ is unbounded, there is a subsequence $\eta''=(n_k'')_{k=1}^\infty$ of $\eta'$ such that $n_k\le n_k''$ for every $k\in\NN$. Then, since $\XXB_p(\delta)$ is a Schauder basis, $\LLB_p[\delta,\eta]\lesssim_c \LLB_p[\delta,\eta'']$. In turn, $\LLB_p[\delta,\eta'']\lesssim_c \LLB_p[\delta,\eta']$ and so $\LLB_p[\delta,\eta]\lesssim_c \LLB_p[\delta,\eta']$.
An appeal to Lemma~\ref{lem:PDC} completes the proof.
\end{proof}
We plan to use Theorem~\ref{thm:Baseslp} for showing the existence of permutatively non-equivalent almost greedy (Schauder) bases of $\ell_p$. To that end we need to find sequences $\delta$ whose associated functions $\Gamma$ have different ratios of growth. Let us start with some examples.

\begin{example}\label{ex:classical} In the classical case $d_n=2$ for all $n\in\NN$, we have $\Lambda(m)=2^{n-1}$ for every $m\in\NN$. Consequently,
\[
\lfloor \log_2(m)\rfloor \le \Gamma(m)\le \lceil \log_2(m)\rceil, \quad m\in\NN.
\]
This way, in light of Theorem~\ref{thm:Baseslp} and Corollary~\ref{cor:EstimateCC}, we obtain a quasi-greedy basis of $\ell_p$ ``as conditional as possible.''
\end{example}

\begin{example}Assume there is $a>0$ such that that $d_n\approx n^{a}$ for $n\in\NN$. Then $\sigma(n)\approx n^{1+a}$ for $n\in \NN$. We infer that there are $n_0\in\NN$ and $0<C_1\le C_2<\infty$ such that $R=\sigma(n_0) > C_2^{1/a}$ and $C_1 n^{1+a} \le \sigma(n) \le C_2 n^{1+a}$ for every $n\ge \NN$. Then, for $n\ge n_0$,
\[
C_1^{\sum_{j=0}^{n-n_0-1} (1+a)^j } R^{(1+a)^{n-n_0}} \le \Lambda(n)\le C_2^{\sum_{j=0}^{n-n_0-1} (1+a)^j } R^{(1+a)^{n-n_0}}.
\]
Since
\[
\sum_{j=0}^{n-n_0-1} (1+a)^j\approx\frac{(1+a)^{n-n_0}-1}{a},\] we infer that
\[
\log_R( \Lambda(n))\approx (1+a)^n, \quad n\ge 0.\]
Hence for $m$ large enough,
\[
\Gamma(m) \approx \log_{1+a} (\log_R( m )) \approx \log (\log ( m )).
\]
\end{example}

The above example hints at the difficulty of estimating the function $\Gamma$ associated to a given sequence $\delta$. Our strategy here will be to start from where we want to arrive, i.e., we will see that under a mild condition on $\Gamma$, there is a sequence $\delta$ whose associated function is (equivalent to) $\Gamma$.

\begin{proposition}\label{eq:ChoosingGamma}Let $(M_n)_{n=0}^\infty$ be an increasing sequence of integers such that $M_0=1$, $M_1=2$, $M_2\ge 4 $, and for $n\in \NN$,
\begin{equation}\label{eq:goodMn}
\left\lceil \frac{M_{n+1}-M_{n}}{M_{n}-M_{n-1}} \right\rceil \le \left\lfloor \frac{M_{n+2}-M_{n+1}}{M_{n+1}-M_{n}} \right\rfloor.
\end{equation}
Then there exists a non-decreasing sequence $\delta$ of integers in $[2,\infty)$ whose associated sequence $\Lambda$ is given by $\Lambda(n)=M_n$ for every $n\in\NN\cup\{0\}$.
\end{proposition}

\begin{proof}
 For $n\in\NN$ and $t\in\RR$, put $f_n(t)= M_n + \lfloor A_n\rfloor(t-M_{n-1})$ and $g_n(t)= M_{n+1}-\lceil A_n\rceil (M_n-t)$, where 
\[
A_n= \frac{M_{n+1}-M_{n}}{M_{n}-M_{n-1}}.
\]
We have $g_n(M_{n-1})\le M_n=f_n(M_{n-1})$ and $f_n(M_n)\le M_{n+1}=g_n(M_n)$. Hence, there is
$h\colon[1,\infty)\to\RR$ such that  for $n\in\NN$ and $M_{n-1}\le t \le M_n$,
\[
h(t)=\max \{ f_n(t), g_n(t)\}, 
\]
and  $h(M_{n-1})=M_n$ for every $n\in\NN$. Since $\lceil A_n \rceil \le \lfloor A_{n+1} \rfloor$ for every $n\in\NN$, $h$ is convex. Therefore, since $h(1)=h(M_0)=M_1=2$, $h(2)=h(M_1)=M_2\ge 4$, and $h(n)\in\ZZ$  for every $n\in\NN$,
 the sequence $\delta=(h(n+1)-h(n))_{n=1}^\infty$ satisfies the desired property.
\end{proof}

\begin{proposition}\label{thm:concave}Let $\phi\colon[0,\infty) \to [0,\infty)$ be an increasing concave function with $\phi(0)=0$. Then there exists a non-decreasing sequence $\delta$ of integers in $[2,\infty)$ whose associated function $\Gamma$ satisfies $\Gamma(m) \approx \phi(\log (m))$ for $m\ge 2$.
\end{proposition}
\begin{proof}
If $\phi(x)\approx x$ for $x\ge 0$, the result follows from Example~\ref{ex:classical}. So, we assume that $\lim_{x\to\infty} \psi(x)/x=\infty$, where $\psi=\phi^{-1}$ is the inverse of $\phi$. Then, since $\psi$ is convex,
\[
\lim_{x\to\infty} \psi(x+1)-\psi(x)=\infty.
\]
Consequently,
\[
F(x):=e^{\psi(x+1)}-e^{\psi(x)}= e^{\psi(x+1)} (1-e^{\psi(x)-\psi(x+1)}) \sim e^{\psi(x+1)}, \quad x\to \infty.
\]
Let $M(x)=\lfloor e^{\psi(x)} \rfloor$, $x\ge 0$. We have
\[
\left| M(x+1)-M(x)- F(x) \right| \le 1, \quad x\ge 0.
\]
We infer that $ M(x+1)-M(x)\sim e^{\psi(x+1)}$ as $x$ goes to $\infty$. Therefore
\[
H(x):=\frac{ M(x+2)-M(x+1)}{M(x+1)-M(x)}
\sim F(x+1)
\sim e^{\psi(x+2) }, \quad x\to\infty.
\]
Iterating the argument we obtain
\[
\frac{H(x+1)}{H(x)} \sim F(x+2) \sim e^{\psi(x+3) }, \quad x\to\infty.
\]
We infer that $\lim_{x\to\infty} H(x+1)-H(x)=\infty$. Consequently, there is $a\in\NN$ such that $2\le M(a+1)-M(a)$ and
\[
5\le \left\lceil \frac{ M(x+2)-M(x+1)}{M(x+1)-M(x)}\right\rceil\le \left\lfloor \frac{M(x+3)-M(x+2)}{M(x+2)-M(x+1)} \right\rfloor, \quad x\ge a.
\]
Let $b\in\NN$ be such that $2^b\le M(a+1)-M(a)<2^{b+1}$. Define
\[
M_n=\begin{cases} 2^n & \text{ if } 0\le n \le b, \\ M_n=M(n-b+a)-M(a)+2^b &\text{ if } b\le n.\end{cases}
\]
It is routine to check that $M_2\ge 4$ and that the sequence $(M_k)_{k=1}^\infty$ satisfies \eqref{eq:goodMn}. Then, by Proposition~\ref{eq:ChoosingGamma} there exists a non-decreasing sequence $\delta$ of integers in $[2,\infty)$ whose associated function $\Gamma$ is the left inverse of $(M_k)_{k=1}^\infty$. If $m\ge 2^b$ and $\Gamma(m)=n$ we have
\[
b-a-1 + \phi(\log ( m+M(a) -2^b +1)) < n< b-a + \phi(\log ( m+M(a) -2^b +1)).
\]
Since the function $\phi\circ\log$ is doubling on the interval $[2,\infty)$, we infer that $\Gamma(m) \approx \phi(\log (m))$ for $m\ge 2$.
\end{proof}

\begin{theorem} \label{thm:twospaces} Let $\YY$ be either the space $\XX_p$ or the space $\ell_p$, $0<p\leq 1$. For every concave increasing function $\phi\colon[0,\infty) \to [0,\infty)$ with $\phi(0)=0$ there exists an almost greedy Schauder basis $\BB$ for $\YY$ with $k_m[\BB, \YY]\approx \phi^{1/p}(\log m)$ for $m\geq 2$.
\end{theorem}
\begin{proof}
For $\XX_p$ it follows from Proposition \ref{thm:concave} and and Theorem \ref{thm:main} and for $\ell_p$ from Proposition \ref{thm:concave} and Theorem \ref{thm:Baseslp}.
\end{proof}

\begin{corollary}\label{cor:potential}
Suppose $0<p\leq 1$. For every $0< a\le 1/p$ the spaces $\XX_p$ and $\ell_p$ contain
an almost greedy basis with conditionality constants $\kk_m \approx (\log m)^a$ for $m\ge 2$.
\end{corollary}

\begin{proof}
Just apply Theorem~\ref{thm:twospaces} with $\phi(x)=x^c$, $0<c\le 1$.
\end{proof}

The following was result was proved for the case $\ell_1$ in \cite{DHK2006} using completly different techniques.

\begin{corollary}\label{dkkplessthan1} Both the spaces $\XX_p$ and $\ell_p$ for $0<p\leq 1$
contain a continuum of mutually permutatively non-equivalent almost greedy bases.
\end{corollary}

\begin{proof}It is immediate from Corollary~\ref{cor:potential}.
\end{proof}
\section{Lindenstrauss dual bases}\label{Lindualbases}
\noindent By Proposition~\ref{prop:monotonebasis}, the sequence $\XXB_p^*(\delta)$ is a Schauder basis of its closed linear span $[\XXB_p^*(\delta)]$ in $\XX_p^*(\delta)$. We will write
\[
\XX_{p,0}^*(\delta) :=[\XXB_p^*(\delta)].
\]
If $0<p<1$ there is not much to say about this basis apart from that, by Proposition~\ref{newpropqenv}, it is equivalent to the unit vector system of $c_0$. If $p=1$, since $\XXB_1(\delta)$ is a conditional basis, so is $\XXB_1^*(\delta)$. In fact, by \cite{AlbiacKalton2016}*{Proposition 3.2.3} and Proposition~\ref{prop:monotonebasis}, we have
\[
\kk_m[\XXB_1^*(\delta), \XX_{1,0}^*(\delta)]=\kk_m[\XXB_1(\delta), \XX_{1}(\delta)]\approx \Gamma(m), \quad m\ge 2.
\]
In this section, we will discuss the basis $\XXB_1^*(\delta)$ and extend the results from \cite{BBGHO2018}, where the authors deal with the classical case $d_{n}=2$.

First, we focus on the Banach space $\XX_{1,0}^*(\delta)$. Note that, since $\XX_1(\delta)$ is a $\SL_1$-space, $\XX_1^*(\delta)$ is isomorphic to $\ell_\infty$ \cite{LinRos1969}. Note also that if $p<1$, by Corollary~\ref{cor:Xpnorminginloo} and Fact~\ref{fact:n}, the dual map $J^*\colon \ell_\infty \to \XX_p^*(\delta)$ restricts to an isomorphism from $c_0$ onto $\XX_{p,0}^*(\delta)$. Next, we show that this result holds even for $p=1$.

\begin{proposition}\label{prop:normingco} For fixed $0<p \le 1$, let $\XX_p=\XX_p(\delta)$, where $\delta=(d_n)_{n=1}^\infty$ is a sequence of integers in $[2,\infty)$. For every $y\in c_{0}$ we have
\[
\Vert y \Vert_{\XX_p^*}= \sup\{|\langle y, x\rangle| \colon x\in B_{\XX_p} \}
\le \Vert y \Vert_\infty \le 2^{1/p} \Vert y \Vert_{\XX_p^*}.
\]
That is, $B_{\XX_p}$ is a norming set for $(c_0,\Vert \cdot\Vert_\infty)$, and the map
$y\mapsto \langle y,\cdot\rangle$ is an isomorphism from $c_0$ onto $ \XX_{p,0}^*(\delta)$.
\end{proposition}

\begin{proof}Given $\varepsilon>0$ and $j\in\NN$, pick $N\in\NN$ such that $\Vert y-S_N(y)\Vert_\infty \le \varepsilon$. Let $1 \le j \le N$. Use Lemma~\ref{lem:2} to choose $x\in\XX_p$ such that $S_N(x)=2^{-1/p}\ee_j$ and $\Vert x\Vert=1$. Then
\begin{align*}
|y(j)|&=2^{1/p} | \langle y,S_N(x)\rangle|\\
&=2^{1/p} | \langle S_N(y),x\rangle|\\
&\le 2^{1/p} | \langle y,x\rangle|+ 2^{1/p} | \langle y-S_N(y),x\rangle|\\
&\le 2^{1/p} \Vert y \Vert_{\XX_p^*}+2^{1/p}\varepsilon.
\end{align*}
Since $\varepsilon$ is arbitrary, taking the supremum on $j$ we obtain the desired inequality.
\end{proof}

Note that, as a consequence of Proposition~\ref{prop:normingco}, the basis $\XXB_1^*(\delta)$ is neither quasi-greedy \cite{DKKT2003}*{Corollary 8.6}, nor super-democratic \cite{AADK}*{Proposition 4.20}. Next, we carry out a quantitive study of this basis.   From now on, as no confusion is possible, we will write  $\XXB_1=\XXB_1(\delta)$, 
$\XXB_1^*=\XXB_1^*(\delta)$, $\XX_1=\XX_1(\delta)$, and
$\XX_{1,0}^*=\XX_{1,0}^*(\delta)$.
We start by estimating the democracy functions and quasi-greedy constants of $\XXB_1^*$.
The $m$th Lebesgue (quasi-greedy) constant $\Leb^{q}_m[\BB,\YY]$ of a basis $\BB$ of a quasi-Banach space $\YY$ will be the smallest constant $C$ such that
\[
\max\{ \Vert \GG_m(f) \Vert , \Vert f- \GG_m(f) \Vert\} \le C \Vert f \Vert, \quad f\in\YY.
\]
\begin{lemma}\label{lem:lowerestimatedualbasis}Let $\delta$ be a non-decreasing sequence of integers in $[2,\infty)$. Then $\XXB^*_{1}$ is not SUCC, therefore it is neither quasi-greedy nor superdemocratic.
Quantitatively, for every $m\in\NN$:
\begin{enumerate}
\item[(i)] $\Phi_m^{l,s}[\XXB_1^*,\XX_{1,0}^*] \le 2$,
\item[(ii)] $\Phi_m^{u}[\XXB_1^*,\XX_{1,0}^*] >\Gamma(m)$, and
\item[(iii)] $\Leb_m^q[\XXB_1^*,\XX_{1,0}^*] \ge \frac{1}{8} \Gamma(m)$.
\end{enumerate}
\end{lemma}
\begin{proof}Let $x_m= \sum_{j=1}^m \xx_j^*$ and $y_m=\sum_{j=1}^m(-1)^j \xx_j^*$. By Proposition~\ref{prop:normingco}, Lemma~\ref{lem:lpnormcolumn}, and Fact~\ref{fact:s},
\[
\left\Vert x_m \right\Vert\ge \frac{1}{2} \left\Vert x_m\right\Vert_\infty\ge \frac{1}{2} \sum_{j=1}^m x_j^*(1) >\frac{1}{2}\Gamma(m).
\]
Let $k\in\NN$. By Fact~\ref{fact:s} the sequence $((-1)^j \xx_j^*(k))_{j=1}^\infty$ is alternating. Since
\[
J_{k,n,m}=J_{k,n}\cap\{1,\dots,m\}, \quad k,m \in\NN ,\, n\in\NN\cup\{0\}
\]
is an integer interval, taking also into account Lemma~\ref{Iproperties}(a), we obtain
\[
\left| \sum_{j\in J_{k,n,m}} (-1)^j \xx_j^*(k) \right|\le A_{k,n}:=\sup_{j\in J_{k,n}} |\xx_j^*(k)|.
\]
Since $(J_{k,n,m})_{n=1}^\infty$ are pairwise disjoint by Fact~\ref{fact:h}, Lemma~\ref{lem:blocksum} (a) yields
\[
\Vert y_m\Vert \le \Vert y_m\Vert_\infty
= \sup_k \left| \sum_{n=0}^\infty \sum_{j\in J_{k,n,m}} (-1)^j \xx_j^*(k) \right|
\le \sup_k \sum_{n=0}^\infty A_{k,n} \le 2.
\]
We infer that $\XXB_{1}^*$ is not SUCC and that, if $C_m=\Leb_m^q[\XXB_1^*,\XX_{1,0}]$,
\[
\Gamma(m) \le 2 \Vert x_m \Vert \le 4 C_m\Vert y_m \Vert \le 8 C_m.\qedhere
\]
\end{proof}
We also provide an estimate for the lower democracy function.
\begin{lemma}Let $\delta$ be a non-decreasing sequence of integers in $[2,\infty)$. Then,
\[\Phi_m^{l}[\XXB_1^*,\XX^*_{1,0}] \le 2,\quad m\in\NN.\]
\end{lemma}

\begin{proof}Let $B=\{\Lambda(n) \colon n\ge 0\}$. It is clear that, for every $k\in\NN$ and $n\ge 0$, $J_{k,n}\cap B$ is either empty or a singleton. By Lemma~\ref{lem:blocksum}(a),
\[
\sum_{j\in B} |\xx_j^*(k)| \le 2, \quad k \in \NN,
\]
whence
\[
\left\Vert \sum_{n\in A} \xx_j^*\right\Vert \le \left\Vert \sum_{n\in A} \xx_j^*\right\Vert_\infty \le 2
\] for every $A\subseteq B$ finite.
\end{proof}

We close our study of the democracy functions of the dual bases of Lindenstrauss bases with an upper bound for the upper superdemocracy constant.
\begin{lemma}\label{lem:bounddualdemocracy}Let $\delta$ be a non-decreasing sequence of integers in $[2,\infty)$. Then,
\[
\Phi_m^{u,s}[\XXB_1^*,\XX_{1,0}^*] \le 1+\Gamma(m) , \quad m\in\NN.
\]
\end{lemma}

\begin{proof}Let $A\subseteq\NN$ with $|A|\le m$ and let $\gamma$ be a sequence of signs. Applying Lemma~\ref{lem:lpnormcolumn} we obtain
\[
\Vert \Ind_{\gamma,A}[\XXB_1^*,\XX_{1,0}^*]\Vert
\le \Vert \Ind_{\gamma,A}[\XXB_1^*,\XX_{1,0}^*]\Vert_\infty
\le \sup_k \sum_{j\in A} |\xx_j^*(k)|\\
\le 1+\Gamma(m).\qedhere
\]
\end{proof}
Given a basis $\BB$ of a quasi-Banach space $\YY$ and $m\in\NN$, the $m$th Lebesgue (almost greedy) constant $\Leb_m^a[\BB,\YY]$ is the optimal constant $C$ such that
\[
\Vert f- \GG_m[\BB,\YY](f) \Vert \le C \Vert f - S_A[\BB,\YY](f) \Vert
\]
for every $A\subseteq \NN$ with $|A|=m$. It is clear that
\[
\left(-1+(\Leb_m^q[\BB,\YY])^p\right)^{1/p} \le \Leb_m^a[\BB,\YY] \le \Leb_m[\BB,\YY].
\]

\begin{theorem}Let $\delta$ be a non-decreasing sequence of integers in $[2,\infty)$. Then for $m\ge 2$,\begin{align*}
\Phi_m^{u}[\XXB_1^*,\XX_{1,0}^*]
&\approx
\Phi_m^{s,u}[\XXB_1^*,\XX_{1,0}^*]
\approx \Leb_m^q[\XXB_1^*,\XX_{1,0}^*]
\approx \Leb_m^a[\XXB_1^*,\XX_{1,0}^*]\\
&\approx \Leb_m [\XXB_1^*,\XX_{1,0}^*]
\approx \kk_m [\XXB_1^*,\XX_{1,0}^*]
\approx \Gamma(m)
\end{align*}
and 
\[
\Phi_m^{l}[\XXB_1^*,\XX_{1,0}^*] \approx \Phi_m^{l,s}[\XXB_1^*,\XX_{1,0}^*]\approx 1.
\]
\end{theorem}

\begin{proof}We only need to prove that $L_m:=\Leb_m [\XXB_1^*,\XX_{1,0}^*]\lesssim \Gamma(m)$ for $m\ge 2$.
Put $D_m= \Phi_m^{u,s}[\XXB_1, \XX_{1}]$, $D_m^*= \Phi_m^{u,s}[\XXB_1^*,\XX_{1,0}^*]$ and $E_m=1+\Gamma(m)$ for $m\in\NN$, and use the convention $D_0^*=E_0=0$. Since $D_m\le 2m$ by Proposition~\ref{prop:dem} and $D_m^*\le E_m$ by Lemma~\ref{lem:bounddualdemocracy}, 
for every $N\in\NN$ we have
\[
\sum_{m=1}^N \frac{ D_m (D_m^*-D_{m-1}^*)}{m} 
\le 2 \sum_{m=1}^N D_m^*-D_{m-1}^*
= 2 D_N^*
\le 2E_N.
\]
Hence, by \cite{BBGHO2018}*{Theorem 1.1},
$
L_m \le 1 +3 E_m
$
for all $m\in\NN$,
i.e.,
\[
L_m\le 4 + 3 \Gamma(m),\quad m\in\NN.\qedhere
\]
\end{proof}


\end{document}